\documentclass[onefignum,onetabnum]{siamart190516}

\usepackage{lipsum}
\usepackage{amsfonts}
\usepackage{graphicx}
\usepackage{epstopdf}
\usepackage{xcolor}
\usepackage{tikz}
\usepackage{algorithmic}
\usepackage{subfloat}
\usepackage{subfig}
\usepackage{enumitem}
\usepackage{comment}

\ifpdf
  \DeclareGraphicsExtensions{.eps,.pdf,.png,.jpg}
\else
  \DeclareGraphicsExtensions{.eps}
\fi

\usepackage{hyperref}
\usepackage{cleveref}
\crefname{equation}{}{}
\crefname{figure}{Fig.}{Figs.}
\crefname{appendix}{}{}
\crefname{table}{Tab.}{Tabs.}
\Crefname{ALC@unique}{Line}{Lines} 

\def\du{[\![}
\def\df{]\!]}
\DeclareMathOperator*{\essinf}{ess\,inf}
\DeclareMathOperator*{\esssup}{ess\,sup}
\def\ol{\overline}
\def\bxi{{\boldsymbol \xi}}

\def\psib{\boldsymbol{\psi}}

\newcommand{\mrm}{\mathsf{m}}

\newcommand{\HO}{\scalebox{.9}{${\scriptscriptstyle\rm H}$}}
\newcommand{\LO}{\scalebox{.9}{${\scriptscriptstyle\rm L}$}}
\newcommand{\CFL}{\scalebox{0.8}{${\rm CFL}$}}

\newcommand{\avg}[1]{\langle #1 \rangle} 

\DeclareMathAlphabet\mathbfcal{OMS}{cmsy}{b}{n}
\newsiamremark{remark}{Remark}
\newsiamremark{hypothesis}{Hypothesis}
\crefname{hypothesis}{Hypothesis}{Hypotheses}
\newsiamthm{claim}{Claim}

\headers{Maximum principle preserving time implicit DGSEM}{F. Renac}

\title{Maximum principle preserving time implicit DGSEM for nonlinear scalar conservation laws}

\author{
 Florent Renac\thanks{DAAA, ONERA, Universit\'e Paris Saclay, F-92322 Ch\^atillon, France ({\tt florent.renac@onera.fr}).} 
}

\ifpdf
\hypersetup{
        pdftitle={Maximum principle preserving time implicit DGSEM for nonlinear scalar conservation laws}
        pdfauthor={Florent Renac}
}
\fi

\begin{document}

\maketitle

\begin{abstract}
This work concerns the analysis of the discontinuous Galerkin spectral element method (DGSEM) with implicit time stepping for the numerical approximation of nonlinear scalar conservation laws in multiple space dimensions. We consider either the DGSEM with a backward Euler time stepping, or a space-time DGSEM discretization to remove the restriction on the time step. We design graph viscosities in space, and in time for the space-time DGSEM, to make the schemes maximum principle preserving and entropy stable for every admissible convex entropy. We also establish well-posedness of the discrete problems by showing existence and uniqueness of the solutions to the nonlinear implicit algebraic relations that need to be solved at each time step. Numerical experiments in one space dimension are presented to illustrate the properties of these schemes.
\end{abstract}




%
%
\section{Introduction}

We are here interested in the accurate and robust approximation of the following Cauchy problem with an hyperbolic scalar nonlinear conservation law in $d\geq1$ space dimensions:
\begin{subequations}\label{eq:hyp_cons_laws}
\begin{align}
 \partial_tu + \nabla\cdot{\bf f}(u) &= 0 \quad \mbox{in }\Omega\times(0,\infty), \label{eq:hyp_cons_laws-a} \\
 u(\cdot,0) &= u_{0}(\cdot)\quad\mbox{in }\Omega, \label{eq:hyp_cons_laws-b} 
\end{align}
\end{subequations}

\noindent with $\Omega\subset\mathbb{R}^d$ and $u_{0}$ in $L^\infty(\mathbb{R}^d,\mathbb{R})$. The solutions to \cref{eq:hyp_cons_laws} may develop discontinuities in finite time and \cref{eq:hyp_cons_laws} has to be understood in the sense of distributions where we look for weak solutions that are piecewise smooth solutions that satisfy jump relations at a surface of discontinuity. Weak solutions are not necessarily unique and \cref{eq:hyp_cons_laws} must be supplemented with further admissibility conditions to select the physical solution. We here focus on entropy inequalities for any admissible convex entropy $\eta(u)$ and smooth entropy flux ${\bf q}(u)=\int^u\eta'(w){\bf f}'(w)dw$ pair:
\begin{equation}\label{eq:PDE_entropy_ineq}
 \partial_t\eta(u) + \nabla\cdot{\bf q}(u) \leq 0 \quad \mbox{in }\Omega\times(0,\infty),
\end{equation}

\noindent in the sense of distributions. This leads to $L^2$ stability for a uniformly convex entropy \cite{Dafermos2016} and in one space dimension, $d=1$, an inequality for one entropy is enough to select the physical weak solution when the flux $f(u)$ is strictly convex \cite{Panov_1994}.


We assume a locally Lipschitz continuous flux ${\bf f}$, and entropy solutions to \cref{eq:hyp_cons_laws} are known to satisfy a maximum principle:
%
\begin{equation}\label{eq:PDE_max_principle}
 \essinf_{{\bf x}\in\Omega}u_0({\bf x})=:m \leq u({\bf x},t) \leq M:=\esssup_{{\bf x}\in\Omega}u_0({\bf x}) \quad \text{ a.e. in }\Omega\times(0,\infty).
\end{equation}

Let us introduce the Lipschitz constant $L_f=L_f(m,M)$ of ${\bf f}$:

\begin{equation}\label{eq:flux_lipschitz_cst}
 L_f := \sup_{m\leq u^+\neq u^-\leq M}\sup_{{\bf n}\in\mathbb{S}^{d-1}}\frac{\big|\big({\bf f}(u^+)-{\bf f}(u^-)\big)\cdot{\bf n}\big|}{|u^+-u^-|}.
\end{equation}

The design of the numerical schemes we consider in this work conveniently relies on the Riemann problem in the unit direction ${\bf n}$ in $\mathbb{S}^{d-1}$: given data $u^-$ and $u^+$, solve

\begin{subequations}\label{eq:Riemann_pb}
\begin{align}
 \partial_tu + \partial_x{\bf f}(u)\cdot{\bf n} &= 0 \quad \mbox{in }\mathbb{R}\times(0,\infty), \label{eq:Riemann_pb-a} \\
 u(x,0) &= u_{0}(x) := 
\left\{ \begin{array}{ll} u^- & \mbox{if }x:={\bf x}\cdot{\bf n}<0, \\ u^+ & \mbox{if }x>0. \end{array}\right. \label{eq:Riemann_pb-b}
\end{align}
\end{subequations}

The maximum wave speed in the solution to \cref{eq:Riemann_pb} is bounded by $L_f(m,M)$ in \cref{eq:flux_lipschitz_cst} with $m=\min(u^-,u^+)$ and $M=\max(u^-,u^+)$. Integrating \cref{eq:Riemann_pb} and the associated entropy inequality over space and time, we have for all $\alpha\geq L_f$ \cite{hll_83,guermond_popov_GV_16} 
\begin{subequations}\label{eq:sol_Riemann_pb}
\begin{align}
 m \leq {\cal W}(u^-,u^+,{\bf n}) &:= \frac{u^-+u^+}{2} - \frac{\big({\bf f}(u^+)-{\bf f}(u^-)\big)\cdot{\bf n}}{2\alpha} \leq M,  \label{eq:sol_Riemann_pb-a} \\
 \eta\big({\cal W}(u^-,u^+,{\bf n})\big) &\leq \frac{\eta(u^-)+\eta(u^+)}{2} - \frac{\big({\bf q}(u^+)-{\bf q}(u^-)\big)\cdot{\bf n}}{2\alpha}.  \label{eq:sol_Riemann_pb-b}
\end{align}
\end{subequations}


We are here interested in the approximation of \cref{eq:hyp_cons_laws} with a high-order discretization that satisfies the above properties at the discrete level. 
We consider the discontinuous Galerkin spectral element method (DGSEM) based on the collocation between interpolation and quadrature points \cite{kopriva_gassner10} and tensor products of one-dimensional function bases and quadrature rules. Using diagonal norm summation-by-parts (SBP) operators and the entropy conservative (EC) numerical fluxes from Tadmor \cite{tadmor87}, semi-discrete EC schemes have been derived in \cite{fisher_carpenter_13,carpenter_etal14} for nonlinear conservation laws and one given entropy. Entropy stable (ES) DGSEM on hexahedral \cite{winters_etal_16,wintermeyer_etal_17} and simplex \cite{chen_shu_17,Crean_etal_SBP_curved_2018}
meshes have been proposed for the given entropy by using the same framework. The particular form of the SBP operators allows to take into account the numerical quadratures that approximate integrals in the numerical scheme compared to other techniques that require their exact evaluation to satisfy the entropy inequality \cite{jiang_shu94,hiltebrand_mishra_14}. The DGSEM thus provides a general framework for the design of ES schemes for nonlinear systems of conservation laws. Numerical experiments in \cite{winters_etal_16,wintermeyer_etal_17,chen_shu_17,renac19,coquel_etal_DGSEM_BN_21} highlight the benefits on stability and robustness of the computations, though this not guarantees to preserve neither the entropy stability at the fully discrete level, nor positivity of the numerical solution. Designs of fully discrete ES and positive DGSEM have been proposed in \cite{despres98,renac17a,renac17b,PAZNER_idg_DGSEM20211,carlier_renac_IDP_22,RUEDARAMIREZ2022105627,Jiang_Lu_IDP_DG_18} among others. Likewise, these schemes are ES for one given entropy only, which may be not sufficient to capture the entropy weak solution as will be observed in the numerical experiments of \cref{sec:num_xp}.

These schemes are usually analyzed in semi-discrete form for which the time derivative is not discretized, or when coupled with explicit in time discretizations. Time explicit integration may however become prohibitive for long time simulations or when looking for stationary solutions due to the strong CFL restriction on the time step which gets smaller as the approximation order of the scheme increases to ensure either linear stability \cite{gassner_kopriva_disp_diss_11,atkins_shu_98,krivodonova_qui_CFL_DG_13}, or positivity of the approximate solution \cite{zhang_shu_10a,zhang2012maximum}. We here consider and analyze a DGSEM discretization in space associated with a time implicit integration to circumvent the CFL condition. We first consider a first-order backward Euler time integration, which constitutes an efficient time stepping when looking for steady-state solutions. We then consider a DGSEM discretization in time more adapted for time-resolved simulations. 

Little is known about the properties of time implicit DGSEM schemes, apart from the entropy stability which holds providing the semi-discrete scheme is ES due to the dissipative character of the backward Euler time integration. The works in \cite{QinShu_impt_positive_DG_18,MRR_BEDGSEM_23} analyze time implicit modal discontinuous Galerkin (DG) and DGSEM for the discretization of a 1D linear scalar hyperbolic equation and show that a lower bound on the time step is required for the cell-averaged solution to satisfy a maximum principle at the discrete level. A linear scaling limiter of the DG solution around its cell-average \cite{zhang_shu_10a} is then used to obtain a maximum principle preserving (MPP) scheme. Unfortunately, these schemes are not MPP in multiple space dimensions even on Cartesian grids neither for the full polynomial solution, nor for the cell-averaged solution \cite{ling_etal_pos_impl_DGM_18,MRR_BEDGSEM_23}. Related works on time implicit DG schemes concern radiative transfer equations with enriched function space \cite{ling_etal_pos_impl_DGM_18}, reduced order quadrature rules \cite{Xu_Shu_PP_DG_2022}, and limiters \cite{Xu_Shu_PP_DG_2022}; or the formulation of a constrained optimization problem \cite{vandervegt_lim_impl_DG_19}.

We aim at designing schemes that satisfy the maximum principle and all entropy inequalities. The MPP property alone may not be enough to capture the entropy weak solution. For instance, the local extrema diminishing method \cite{jameson_LED_95} is MPP, but may not converge to the entropy solution \cite[Lemma~3.2]{guermond_popov_IDP_CFE_scalar_17}. We here add a first-order graph viscosity to the DGSEM schemes. Graph viscosity is a common method of artificial dissipation based on the connectivity graph of the degrees of freedom (DOFs) \cite{jameson_LED_95,guermond_popov_GV_16,Guermond_etal_IDP_conv_lim_18,PAZNER_idg_DGSEM20211}. The graph viscosity is an efficient technique to impose invariant domain properties to the scheme such as positivity, maximum principle, or minimum entropy principle. 
Some limiters in high-order schemes also use low-order schemes with graph viscosity as building-blocks such as the flux-corrected transport limiter \cite{BORIS_Book_FCT_73,zalesak1979fully,Guermond_IDP_NS_2021,ern_guermond_IDP_DIRK_23,MRR_BEDGSEM_23}, or convex limiting \cite{Guermond_etal_IDP_conv_lim_18,PAZNER_idg_DGSEM20211}.

The graph viscosity is here based on the connectivity graph of the DOFs either within the mesh element when the DGSEM is associated to the backward-Euler time integration, or within the space-time discretization element for the space-time DGSEM. It is local to those elements and do not use connectivity with neighboring elements making its implementation straightforward. We derive estimates on the viscosity coefficients to make the schemes MPP and ES for all admissible convex entropies. We also analyze the well-posedness of the discrete problems and prove existence and uniqueness of the solutions to the nonlinear implicit discrete problems that need to be solved at each time step. One key ingredient to prove these properties is the derivation in \cref{th:hec_over_fan,th:ineq_oleta} of discrete counterparts to \cref{eq:sol_Riemann_pb} for the EC two-point fluxes in space. Though, those discrete counterparts are well known for positive and entropic three-point schemes and correspond to consistency relations with the integral forms of \cref{eq:hyp_cons_laws-a,eq:PDE_entropy_ineq} \cite[Th.~ 3.1]{hll_83}, they constitute new results for EC two-point fluxes up to our knowledge.

The paper is organized as follows. \Cref{sec:DGSEM_discr} describes the DGSEM discretization in space, recall some properties of the metric terms, and of the numerical fluxes. The properties of the DGSEM with a backward Euler time stepping and graph viscosity in space are analyzed in \cref{sec:DGSEM_2D_BE}, while the space-time DGSEM with space and time graph viscosity is described and analyzed in \cref{sec:space_time_DGSEM_2D}. The results are assessed by numerical experiments in one space dimension in \cref{sec:num_xp} and concluding remarks about this work are given in \cref{sec:conclusions}.

%
%
\section{DGSEM approximation in space}\label{sec:DGSEM_discr}

We here describe the semi-discrete weak formulation of problem \cref{eq:hyp_cons_laws} with the DGSEM. The domain is discretized with a shape-regular mesh $\Omega_h\subset\mathbb{R}^d$ consisting of nonoverlapping and nonempty open elements $\kappa$ and we assume that it forms a partition of $\Omega$. By $h:=\sup_{\kappa\in\Omega_h}\text{diam}\,\kappa$ we denote the mesh size. For the sake of clarity, we introduce the DGSEM in two space dimensions $d=2$, its extension (resp., restriction) to $d=3$ (resp., $d=1$) being straightforward.

\subsection{The DGSEM approximation in space}

We look for approximate solutions in the function space of discontinuous polynomials
\begin{equation*}
 {\cal V}_h^p=\{\phi\in L^2(\Omega_h):\;\phi|_{\kappa}\circ{\bf x}_\kappa\in{\cal Q}^p(I^2)\; \forall\kappa\in \Omega_h\},
\end{equation*}

\noindent where ${\cal Q}^p(I^2)$ denotes the space of functions over the master element $I^2:=\{\bxi=(\xi,\eta):\;-1\leq\xi,\eta\leq1\}$ formed by tensor products of polynomials of degree $p\geq1$ in each direction. Each physical element $\kappa$ is the image of $I^2$ through the mapping ${\bf x}={\bf x}_\kappa(\bxi)$. Likewise, each edge $e$ is the image of $I=[-1,1]$ through the mapping ${\bf x}={\bf x}_e(\xi)$. The approximate solution to \cref{eq:hyp_cons_laws} is sought under the form
\begin{equation}\label{eq:DGSEM_num_sol}
 u_h({\bf x},t)=\sum_{i,j=0}^p\phi_\kappa^{ij}({\bf x})U_\kappa^{ij}(t) \quad \forall{\bf x}\in\kappa,\, \kappa\in \Omega_h,\, \forall t\geq0,
\end{equation}

\noindent where $(U_\kappa^{ij})_{0\leq i,j\leq p}$ are the DOFs in the element $\kappa$. The subset $(\phi_\kappa^{ij})_{0\leq i,j\leq p}$ constitutes a basis of ${\cal V}_h^p$ restricted onto the element $\kappa$ and $(p+1)^2$ is its dimension. Let $(\ell_k)_{0\leq k\leq p}$ be the Lagrange interpolation polynomials in one space dimension associated to the Gauss-Lobatto quadrature nodes over $I$, $\xi_0=-1<\xi_1<\dots<\xi_p=1$:
\begin{equation}\label{eq:cardinalty_Lag_polynom}
 \ell_k(\xi_l)=\delta_{kl}, \quad 0\leq k,l \leq p,
\end{equation}

\noindent with  $\delta_{kl}$ the Kronecker delta. We use tensor products of these polynomials and of Gauss-Lobatto nodes (see \cref{fig:stencil_2D_DGSEM}):
\begin{equation}\label{eq:lag_basis}
 \phi_\kappa^{ij}({\bf x})=\phi_\kappa^{ij}({\bf x}_\kappa(\bxi))=\ell_i(\xi)\ell_j(\eta), \quad 0\leq i,j\leq p.
\end{equation}

\noindent which satisfy the following relation at quadrature points $\bxi_{i'j'}=(\xi_{i'},\xi_{j'})$ in $I^2$:
\begin{equation*}
 \phi_\kappa^{ij}({\bf x}_\kappa^{i'j'})=\delta_{ii'}\delta_{jj'}, \quad 0\leq i,j,i',j' \leq p, \quad {\bf x}_\kappa^{i'j'}:={\bf x}_\kappa(\bxi_{i'j'}),
\end{equation*}

\noindent so the DOFs correspond to point values of the solution: $U_\kappa^{ij}(t)=u_h({\bf x}_\kappa^{ij},t)$.

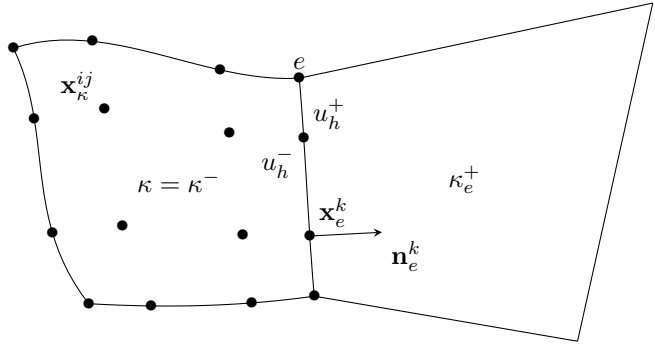
\begin{figure}[ht]
\begin{center}
\begin{tikzpicture}
[declare function={la(\x) = (\x+1/sqrt(5))/(-1+1/sqrt(5)) * (\x-1/sqrt(5))/(-1-1/sqrt(5)) * (\x-1)/(-2);
                   lb(\x) = (\x+1)/(-1/sqrt(5)+1) * (\x-1/sqrt(5))/(-2/sqrt(5)) * (\x-1)/(-1/sqrt(5)-1);
                   lc(\x) = (\x+1)/(1/sqrt(5)+1) * (\x+1/sqrt(5))/(2/sqrt(5)) * (\x-1)/(1/sqrt(5)-1);
                   ld(\x) = (\x+1)/(2) * (\x+1/sqrt(5))/(1+1/sqrt(5)) * (\x-1/sqrt(5))/(1-1/sqrt(5)); 
                   }]
%
%
\draw (1.2,1.625) node {$\kappa=\kappa^-$};
\draw (5.,1.65)   node {$\kappa_e^+$};
\draw (2.8,3.)    node[above] {$e$};
\draw (2.86,2.20) node[below left]  {$u_h^-$};
\draw (2.86,2.20) node[above right] {$u_h^+$};
\draw [>=stealth,->] (2.94,0.90) -- (3.9,0.95) ;
\draw (3.9,0.95) node[below right] {${\bf n}_e^k$};
\draw (2.94,0.90) node[above right] {${\bf x}_e^k$};
\draw (0.21,2.59) node[above left] {${\bf x}_\kappa^{ij}$};
\def\xad{-1.00};  \def\yad{3.40}; \def\xbd{0.05}; \def\ybd{3.49}; \def\xcd{1.75}; \def\ycd{3.11}; \def\xdd{2.80}; \def\ydd{3.00};
\def\xac{-0.723}; \def\yac{2.46}; \def\xbc{0.21}; \def\ybc{2.59}; \def\xcc{1.87}; \def\ycc{2.27}; \def\xdc{2.86}; \def\ydc{2.20};
\def\xab{-0.48};  \def\yab{0.94}; \def\xbb{0.45}; \def\ybb{1.03}; \def\xcb{2.05}; \def\ycb{0.91}; \def\xdb{2.94}; \def\ydb{0.90};
\def\xaa{0.00};   \def\yaa{0.00}; \def\xba{0.83}; \def\yba{-0.03}; \def\xca{2.17}; \def\yca{0.01}; \def\xda{3.00}; \def\yda{0.10};
\draw (\xad,\yad) node {$\bullet$}; \draw (\xbd,\ybd) node {$\bullet$}; \draw (\xcd,\ycd) node {$\bullet$}; \draw (\xdd,\ydd) node {$\bullet$}; 
\draw (\xac,\yac) node {$\bullet$}; \draw (\xbc,\ybc) node {$\bullet$}; \draw (\xcc,\ycc) node {$\bullet$}; \draw (\xdc,\ydc) node {$\bullet$}; 
\draw (\xab,\yab) node {$\bullet$}; \draw (\xbb,\ybb) node {$\bullet$}; \draw (\xcb,\ycb) node {$\bullet$}; \draw (\xdb,\ydb) node {$\bullet$}; 
\draw (\xaa,\yaa) node {$\bullet$}; \draw (\xba,\yba) node {$\bullet$}; \draw (\xca,\yca) node {$\bullet$}; \draw (\xda,\yda) node {$\bullet$}; 
\draw [domain=-1:1] plot ({la(\x)*\xaa+lb(\x)*\xba+lc(\x)*\xca+ld(\x)*\xda}, {la(\x)*\yaa+lb(\x)*\yba+lc(\x)*\yca+ld(\x)*\yda});
\draw [domain=-1:1] plot ({la(\x)*\xad+lb(\x)*\xbd+lc(\x)*\xcd+ld(\x)*\xdd}, {la(\x)*\yad+lb(\x)*\ybd+lc(\x)*\ycd+ld(\x)*\ydd});
\draw [domain=-1:1] plot ({la(\x)*\xda+lb(\x)*\xdb+lc(\x)*\xdc+ld(\x)*\xdd}, {la(\x)*\yda+lb(\x)*\ydb+lc(\x)*\ydc+ld(\x)*\ydd});
\draw [domain=-1:1] plot ({la(\x)*\xaa+lb(\x)*\xab+lc(\x)*\xac+ld(\x)*\xad}, {la(\x)*\yaa+lb(\x)*\yab+lc(\x)*\yac+ld(\x)*\yad});
%
\draw [>=stealth,-] (3.,0.1) -- (6.5,-0.5) ;
\draw [>=stealth,-] (6.5,-0.5) -- (7.5,4.) ;
\draw [>=stealth,-] (7.5,4.) -- (2.8,3.) ;
\end{tikzpicture}
\caption{Inner and outer elements, $\kappa^-$ and $\kappa_e^+$, for $d=2$; definitions of traces $u_h^\pm$ on the interface $e$ and of the unit outward normal vector ${\bf n}_e^k={\bf n}_e({\bf x}_e^k)$ with ${\bf x}_e^k={\bf x}_e(\xi_k)$; positions of quadrature points in $\kappa^-$, ${\bf x}_\kappa^{ij}={\bf x}_\kappa(\bxi_{ij})$, and on $e$, ${\bf x}_e^k$, for $p=3$ (bullets $\bullet$). The elements $\kappa$ are interpolated on the same grid of quadrature points as the numerical solution, i.e., ${\bf x}_\kappa(\bxi)=\sum_{i,j=0}^p\ell_i(\xi)\ell_j(\eta){\bf x}_\kappa^{ij}$.}
\label{fig:stencil_2D_DGSEM}
\end{center}
\end{figure}

The integrals over elements and faces are approximated by using the Gauss-Lobatto quadrature rules so the quadrature and interpolation nodes are collocated:
\begin{equation}\label{eq:GaussLobatto_quad}
 \int_{-1}^1 f(\xi)d\xi \simeq \sum_{k=0}^p \omega_k f(\xi_k),
\end{equation}

\noindent with $\omega_k>0$, satisfying $\sum_{k=0}^p\omega_k=\int_{-1}^1ds=2$, the weights and $\xi_k$ the nodes of the quadrature rule over $I$.

The semi-discrete form of the DGSEM in space is now standard and we refer to \cite{wintermeyer_etal_17} and references therein for details on its implementation on unstructured grids with high-order mesh elements. It reads 
\begin{equation}\label{eq:semi-discr_DGSEM}
 \omega_i\omega_j J_\kappa^{ij}\frac{dU_\kappa^{ij}}{dt} + R_\kappa^{ij}(u_h) = 0 \quad \forall\kappa\in \Omega_h,\; 0\leq i,j\leq p,\; t>0,
\end{equation}

\noindent with $J_\kappa^{ij}:=\det{\bf x}_\kappa'(\xi_i,\eta_j)>0$ and
\begin{align}
  R_\kappa^{ij}(u_h) &= 2\Big(\omega_j\sum_{k=0}^p Q^p_{ik}h_{ec}\big(U_\kappa^{ij},U_\kappa^{kj},{\bf n}_\kappa^{(ik)j}\big) + \omega_i\sum_{k=0}^p Q^p_{jk}h_{ec}\big(U_\kappa^{ij},U_\kappa^{ik},{\bf n}_\kappa^{i(jk)}\big)\Big) \nonumber \\
   &+ \sum_{e\in\partial\kappa}\sum_{k=0}^p \phi_\kappa^{ij}({\bf x}_e^k)\omega_kJ_e^k\Big(h\big(U_\kappa^{ij},u_h^{+}({\bf x}_e^{k},t),{\bf n}_e^k\big)-{\bf f}(U_\kappa^{ij})\cdot{\bf n}_e^k\Big), \label{eq:semi-discr_DGSEM-res}
\end{align}

\noindent where ${\bf n}_e^k={\bf n}_e({\bf x}_e^k)$, $J_e^k=\det{\bf x}_e'(\xi_k)>0$, and by \cref{eq:cardinalty_Lag_polynom}, $\phi_\kappa^{ij}({\bf x}_e^k)=1$ if ${\bf x}_\kappa^{ij}={\bf x}_e^k$ and $\phi_\kappa^{ij}({\bf x}_e^k)=0$ otherwise, while the two first arguments in $h(\cdot,\cdot,\cdot)$ correspond to the traces, $u_h^\mp({\bf x}_e^k,t)$, of the numerical solution at ${\bf x}_e^k$ on $e$ (see \cref{fig:stencil_2D_DGSEM}). The quantities ${\bf n}_\kappa^{(ik)j}$ and ${\bf n}_\kappa^{i(jk)}$ are defined in \cref{eq:def_Jkappa_mean} and have been introduced to keep conservation of the scheme \cite{wintermeyer_etal_17}.

\subsection{Metric terms}

The coefficients $Q^p_{kl}$ in \cref{eq:semi-discr_DGSEM-res} are defined from the entries of the discrete derivative matrix, $D^p_{kl}$, \cite{kopriva_book} as follows:
\begin{equation}\label{eq:nodalGL_diff_matrix}
 D^p_{kl} = \ell_l'(\xi_k), \quad Q^p_{kl}=\omega_kD^p_{kl}, \quad 0\leq k,l \leq p,
\end{equation}

\noindent where the exponent $p$ refers to the polynomial degree. As noticed in \cite{gassner_13}, the DGSEM satisfies the summation-by-parts (SBP) property \cite{STRAND_SBP_94}:

\begin{equation}\label{eq:SBP}
 Q^p_{kl}+Q^p_{lk}=\delta_{kp}\delta_{lp}-\delta_{k0}\delta_{l0}, \quad 0\leq k,l \leq p,
\end{equation}

\noindent which is the discrete counterpart to integration by parts. We will also use the following relations of partition of unity and integration of the Lagrange polynomial derivatives:
\begin{equation}\label{eq:interp_lag_unite_deriv}
 \sum_{l=0}^pQ^p_{kl} = \omega_k \sum_{l=0}^pD^p_{kl} = 0, \quad \sum_{l=0}^pQ^p_{lk} = \delta_{kp}-\delta_{k0} , \quad 0\leq k\leq p;
\end{equation}

\noindent and the metric identities \cite{kopriva_metric_id_06}, $\sum_{i=1}^d\partial_{\xi_i}\big(J_\kappa\nabla\xi_i\big)=0$, which may be written for $d=2$ at the discrete level as
\begin{equation*}
 \sum_{k=0}^p D^p_{ik}J_\kappa^{kj}\nabla\xi({\bf x}_\kappa^{kj}) + D^p_{jk}J_\kappa^{ik}\nabla\eta({\bf x}_\kappa^{ik}) = 0 \quad \forall \kappa\in\Omega_h,\; 0\leq i,j \leq p,
\end{equation*}

Combining the above relation with \cref{eq:interp_lag_unite_deriv} we get
\begin{equation}\label{eq:discr_Met_Id}
 \sum_{k=0}^p D^p_{ik}{\bf n}_\kappa^{(ik)j} + D^p_{jk}{\bf n}_\kappa^{i(jk)} = 0 \quad \forall \kappa\in\Omega_h,\; 0\leq i,j \leq p,
\end{equation}

\noindent where 
\begin{equation}\label{eq:def_Jkappa_mean}
 {\bf n}_\kappa^{(ik)j}=\frac{1}{2}\big(J_\kappa^{ij}\nabla\xi(\bxi_{ij})+J_\kappa^{kj}\nabla\xi(\bxi_{kj})\big), \; {\bf n}_\kappa^{i(jk)}=\frac{1}{2}\big(J_\kappa^{ij}\nabla\eta(\bxi_{ij})+J_\kappa^{ik}\nabla\eta(\bxi_{ik})\big).
\end{equation}

Further assuming the mesh is watertight \cite[App.~B.2]{HENNEMANN_etal_FVGDSEM_21}, i.e., the geometry is continuous across faces, the volume and face metric terms are related by the following relations when evaluated at faces (with some slight abuse in the notations, but without ambiguity)
\begin{subequations}\label{eq:link_vol_surf_metric}
\begin{align}
    {\bf n}_\kappa^{(pp)j} &=J_e({\bf x}_\kappa^{pj}){\bf n}_e({\bf x}_\kappa^{pj}), \quad {\bf n}_\kappa^{(00)j}=-J_e({\bf x}_\kappa^{0j}){\bf n}_e({\bf x}_\kappa^{0j}) \quad \forall 0\leq j\leq p, \\
    {\bf n}_\kappa^{i(pp)} &= J_e({\bf x}_\kappa^{ip}){\bf n}_e({\bf x}_\kappa^{ip}), \quad {\bf n}_\kappa^{i(00)}=-J_e({\bf x}_\kappa^{i0}){\bf n}_e({\bf x}_\kappa^{i0}) \quad \forall 0\leq i\leq p,
\end{align}
\end{subequations}

\noindent where ${\bf n}_e$ denotes the unit normal to $e$ in $\partial\kappa$ pointing outward from $\kappa$ (see \cref{fig:stencil_2D_DGSEM}).


We assume that the mesh is regular in the following sense: $h:=\sup_{\kappa\in\Omega_h}\text{diam}\,\kappa<\infty$ and there exists a finite $\beta=\beta(p,h)>0$ such that 
\begin{equation}\label{eq:regular_mesh_cond}
 J_\kappa^{ij} \geq \frac{\beta}{2d} h^d, \quad \frac{1}{4d\beta}h^{d-1}\geq J_e^k,|{\bf n}_\kappa^{(ik)j}|,|{\bf n}_\kappa^{i(jk)}|>0 \quad \forall \kappa \in\Omega_h, \; 0\leq i,j,k\leq p,
\end{equation}

\noindent which enforces classical regularity assumptions \cite[\S~24]{EYMARD2000713}: $|\kappa|:=\sum_{0\leq i,j\leq p}\omega_i\omega_jJ_\kappa^{ij}\geq \beta h^d$ and $|\partial\kappa|=\sum_{e\in\partial\kappa}\sum_{k=0}^p\omega_kJ_e^k\leq \tfrac{1}{\beta}h^{d-1}$.

Finally, we define the cell-average operator:
\begin{equation}\label{eq:cell-average}
 \langle u_h\rangle_\kappa(t) := \sum_{0\leq i,j\leq p}\omega_i\omega_j\frac{J_\kappa^{ij}}{|\kappa|}U_\kappa^{ij}(t).
\end{equation}

\subsection{Two-point numerical fluxes}\label{sec:two_point_fluxes}
\subsubsection{Cell two-point fluxes}
The numerical flux in the cells is smooth, consistent, $h_{ec}(u,u,{\bf n})={\bf f}(u)\cdot{\bf n}$, symmetric, $h_{ec}(u^-,u^+,{\bf n})=h_{ec}(u^+,u^-,{\bf n})$, linear in ${\bf n}$, and is EC \cite{tadmor87} for a given pair $\big(\vartheta(u),{\bf g}(u)\big)$ with $\vartheta$ strictly convex: 
\begin{equation}\label{eq:entropy_conserv_flux}
 \big(v(u^+)-v(u^-)\big) h_{ec}(u^-,u^+,{\bf n})  = \big(\psib(u^+)-\psib(u^-)\big)\cdot{\bf n} \quad \forall u^\pm,
\end{equation}

\noindent with $v(u)=\vartheta'(u)$ the entropy variable, $\psib(u)=v(u){\bf f}(u) - {\bf g}(u)$ the entropy flux potential, and ${\bf n}$ in $\mathbb{R}^d$. \Cref{th:hec_over_fan} states discrete counterparts to stability properties \cref{eq:sol_Riemann_pb} for the exact Riemann solution \cref{eq:Riemann_pb}. Together with \cref{th:ineq_oleta}, they are key results to establish the properties of the numerical schemes in \cref{sec:DGSEM_2D_BE,sec:space_time_DGSEM_2D}.

\begin{lemma}\label{th:hec_over_fan}
The EC flux \cref{eq:entropy_conserv_flux} for the entropy pair $\big(\vartheta(u),{\bf g}(u)\big)$ with $\vartheta''(\cdot)>0$ satisfies the following relations for all $\alpha\geq L_f$ and all convex entropy pairs $(\eta,{\bf q})$:

\begin{subequations}\label{eq:hec_over_fan}
\begin{align}
 m \leq {\cal U}(u^-,u^+,{\bf n}) &:= (1-\beta)u^-+\beta u^+ - \frac{h_{ec}(u^-,u^+,{\bf n})-{\bf f}(u^-)\cdot{\bf n}}{2\alpha|{\bf n}|} \leq M,  \label{eq:hec_over_fan-a} \\
 \eta\big({\cal U}(u^-,u^+,{\bf n})\big) &\leq \frac{\eta(u^-)+\ol{\eta}\big(u^-,u^+)}{2} - \frac{q_{ec}(u^-,u^+,{\bf n})-{\bf q}(u^-)\cdot{\bf n}}{2\alpha|{\bf n}|},  \label{eq:hec_over_fan-b}
\end{align}
\end{subequations}

\noindent with
\begin{subequations}
\begin{align} 
 \beta=\beta(u^-,u^+) &=\frac{1}{2}\int_0^1\frac{u\big(\theta v^++(1-\theta)v^-\big)-u^-}{u^+-u^-}d\theta \in \Big(0,\frac{1}{2}\Big], \label{eq:def_beta_in_Uec} \\
 \ol{\eta}(u^-,u^+) &=\int_0^1 \eta\circ u\big(\theta v^++(1-\theta)v^-\big)d\theta,  \label{eq:def_ol_eta} \\
 q_{ec}(u^-,u^+,{\bf n}) &=\int_0^1{\bf q}\circ u\big(\theta v^++(1-\theta)v^-\big)\cdot{\bf n}d\theta, \label{eq:def_qec}
\end{align}
\end{subequations}
\noindent where $v^\pm=v(u^\pm)=\vartheta'(u^\pm)$.
\end{lemma}

\begin{proof}
Note that \cref{th:hec_over_fan} holds when $u^-=u^+=u$ since ${\cal U}(u,u,{\bf n})=u$, $\beta(u,u)=\tfrac{1}{2}\int_0^1\theta d\theta=\tfrac{1}{4}$, and $h_{ec}(u,u,{\bf n})={\bf f}(u)\cdot{\bf n}$. We now assume $u^-\neq u^+$. Using the entropy variable $v(u)=\vartheta'(u)$, we have $\psib\circ u(v)=v{\bf f}\big(u(v)\big)-\int^v w{\bf f}'\big(u(w)\big)du(w)=\int^v{\bf f}\circ u(w)dw$ so the EC flux \cref{eq:entropy_conserv_flux} has the closed form
$$h_{ec}(u^-,u^+,{\bf n})=\int_0^1{\bf f}\circ u\big(\theta v^++(1-\theta)v^-\big)\cdot{\bf n}d\theta.$$

\noindent By strict convexity of $\vartheta(u)$, $u=u(v)$ is strictly increasing so $u\big(\theta v^++(1-\theta)v^-\big)$ lies in $[\min(u^\pm), \max(u^\pm)]\subset[m,M]$ and \cref{eq:def_beta_in_Uec} follows directly. Using \cref{eq:def_beta_in_Uec}, we rewrite \cref{eq:hec_over_fan-a} as
\begin{equation*}
 {\cal U}(u^-,u^+,{\bf n}) \!=\! \int_0^1 \frac{u\big(\theta v^++(1-\theta)v^-\big)+u^-}{2}-\frac{{\bf f}\circ u\big(\theta v^++(1-\theta)v^-\big)-{\bf f}(u ^-)}{2\alpha|{\bf n}|}\cdot{\bf n}d\theta
\end{equation*}

\noindent which lies in $[m,M]$ from \cref{eq:sol_Riemann_pb-a}. Likewise, using the Jensen's inequality, then \cref{eq:sol_Riemann_pb-b}, we obtain
%
\begin{multline*}
 \eta({\cal U}(u^-,u^+,{\bf n})) \\ \leq \int_0^1 \eta\bigg(\frac{u\big(\theta v^++(1-\theta)v^-\big)+u^-}{2}-\frac{{\bf f}\circ u\big(\theta v^++(1-\theta)v^-\big)-{\bf f}(u ^-)}{2\alpha|{\bf n}|}\cdot{\bf n}\bigg)d\theta \\
 \leq \int_0^1  \frac{\eta\circ u\big(\theta v^++(1-\theta)v^-\big)+\eta(u^-)}{2} -\frac{{\bf q}\circ u\big(\theta v^++(1-\theta)v^-\big)-{\bf q}(u ^-)}{2\alpha|{\bf n}|}\cdot{\bf n}d\theta
\end{multline*}

\noindent which completes the proof.
\end{proof}

\begin{lemma}\label{th:ineq_oleta}
Under the assumptions of \cref{th:hec_over_fan}, $\beta(u^-,u^+)$ in \cref{eq:def_beta_in_Uec} and $\ol{\eta}(u^-,u^+)$ in \cref{eq:def_ol_eta} satisfy
\begin{equation}\label{eq:ineq_oleta}
 \ol{\eta}(u^-,u^+) - \eta(u^-) \leq 2\beta(u^-,u^+)\big(\eta(u^+) - \eta(u^-)\big).
\end{equation}
\end{lemma}

\begin{proof}
Set again $v^\pm=v(u^\pm)=\vartheta'(u^\pm)$, \cref{eq:ineq_oleta} trivially holds for $u^+=u^-$ and we now assume $u^+\neq u^-$, thus $v^+\neq v^-$. Use the change of variables $u\big(\theta v^++(1-\theta)v^-)=\varsigma u^++(1-\varsigma)u^-$ in \cref{eq:def_ol_eta}, which is one-to-one since $u(v)$ is strictly increasing by convexity of $\vartheta(u)$ in \cref{th:hec_over_fan}, to obtain
\begin{align*}
 \ol{\eta}(u^-,u^+) &= \frac{u^+-u^-}{v^+-v^-}\int_0^1\frac{\eta\big(\varsigma u^++(1-\varsigma)u^-\big)}{u'\big(v\big(\varsigma u^++(1-\varsigma)u^-\big)\big)}d\varsigma \\
 &\leq \frac{u^+-u^-}{v^+-v^-}\int_0^1\frac{\varsigma\eta(u^+)+(1-\varsigma)\eta(u^-)}{u'\big(v\big(\varsigma u^++(1-\varsigma)u^-\big)\big)}d\varsigma \\
 &= 2\beta(u^-,u^+)\eta(u^+) + \big(1-2\beta(u^-,u^+)\big)\eta(u^-),
\end{align*}
\noindent where we have used convexity of $\eta(\cdot)$, then $2\beta(u^-,u^+)=\tfrac{u^+-u^-}{v^+-v^-}\int_0^1\frac{\varsigma d\varsigma}{u'(v(\varsigma u^++(1-\varsigma)u^-))}$ with the same change of variables. This completes the proof.
\end{proof}

\begin{remark}[square entropy]\label{rk:square_entropy1}
For the square entropy, $\vartheta(u)=\tfrac{u^2}{2}$, the entropy flux reads ${\bf g}(u)=\int^uw {\bf f}'(w)dw$, and we have $v(u)=u$, $h_{ec}(u^-,u^+,{\bf n})=\int_0^1{\bf f}(\theta u^++(1-\theta)u^-)\cdot{\bf n}d\theta$, $\beta(u^-,u^+)=\tfrac{1}{4}$ in \cref{th:hec_over_fan}, while $\ol{\eta}(u^-,u^+)=\int_0^1\eta(\theta u^++(1-\theta)u^-)d\theta\leq\tfrac{\eta(u^+)+\eta(u^-)}{2}$ in agreement with \cref{th:ineq_oleta}.
\end{remark}

\subsubsection{Interface two-point fluxes}

The numerical flux at interfaces is Lipschitz continuous, consistent, $h(u,u,{\bf n}) = {\bf f}(u)\cdot{\bf n}$, conservative, $h(u^-,u^+,{\bf n}) =-h(u^+,u^-,-{\bf n})$, and monotone: $h(u^-,u^+,{\bf n})$ is nondecreasing with $u^-$ and nonincreasing with $u^+$. The Lipschitz constants of $h$ with respect to its two first arguments satisfy \cref{eq:flux_lipschitz_cst}: $L_1+L_2\geq L_f$. The associated three-point scheme is therefore MPP and ES \cite[ch.~4]{godlewski_raviart_book1991}: for all $m\leq u^-,u,u^+ \leq M$, ${\bf n}\in\mathbb{S}^{d-1}$, and $\alpha\geq L_f$, we have
\begin{equation}\label{eq:MPP_interface_flux}
m \leq u - \frac{h(u,u^+,{\bf n}) - h(u^-,u,{\bf n})}{2\alpha} \leq M, 
\end{equation}

\noindent and
\begin{equation}\label{eq:ES_interface_flux}
\eta\Big(u - \frac{h(u,u^+,{\bf n}) - h(u^-,u,{\bf n})}{2\alpha}\Big) \leq \eta(u) - \frac{Q(u,u^+,{\bf n}) - Q(u^-,u,{\bf n})}{2\alpha}, 
\end{equation}

\noindent with the associated consistent, $Q(u,u,{\bf n})={\bf q}(u)\cdot{\bf n}$, and conservative, $Q(u^-,u^+,{\bf n})=-Q(u^+,u^-,-{\bf n})$, numerical entropy flux. The Godunov \cite{godunov_59}, Engquist-Osher \cite{engquist_osher_1981} and the Rusanov \cite{FriedrichsLax_71} schemes satisfy the above properties.

%
%
\section{Time implicit discretization and graph viscosity}\label{sec:DGSEM_2D_BE}

We now focus on a time implicit discretization with a backward Euler method in \cref{eq:semi-discr_DGSEM}. The fully discrete scheme reads
\begin{equation}\label{eq:fully-discr_DGSEM}
  \omega_i\omega_j J_\kappa^{ij} \frac{U_\kappa^{ij,n+1}-U_\kappa^{ij,n}}{\Delta t^{(n)}} + R_\kappa^{ij}(u_h^{(n+1)}) + V_\kappa^{ij}(u_h^{(n+1)}) = 0, \; \kappa\in \Omega_h, \; 0\leq i,j\leq p, \; n\geq 0,
\end{equation} 

\noindent where the residuals $R_\kappa^{ij}(\cdot)$ are defined by \cref{eq:semi-discr_DGSEM-res}, $\Delta t^{(n)}=t^{(n+1)}-t^{(n)}>0$, with $t^{(0)}=0$, denotes the time step, and $U_\kappa^{ij,n}=U_\kappa^{ij}(t^{(n)})$. The projection of the initial condition \cref{eq:hyp_cons_laws-b} onto the function space reads $U_\kappa^{ij,0} = u_0({\bf x}_\kappa^{ij})$.

The artificial viscosity term is a graph viscosity \cite{guermond_popov_GV_16} of the form
\begin{equation}\label{eq:graph_viscosity}
 V_\kappa^{ij}(u_h) := d_\kappa \omega_i\omega_j\sum_{k=0}^p\frac{\omega_k}{2}|{\bf n}_\kappa^{(ik)j}|(U_\kappa^{ij}-U_\kappa^{kj})+\frac{\omega_k}{2}|{\bf n}_\kappa^{i(jk)}|(U_\kappa^{ij}-U_\kappa^{ik}).
\end{equation}

One recovers the unlimited time implicit high-order in space scheme when $d_\kappa=0$ for all $\kappa\in\Omega_h$. The scheme with $d_\kappa>0$ is formally first-order accurate in space. We will look for conditions on $d_\kappa$ for the scheme \cref{eq:fully-discr_DGSEM} to satisfy the desired properties.

The method is conservative in the sense that by summing \cref{eq:fully-discr_DGSEM} over $0\leq i,j\leq p$ gives for the cell-averaged solution \cref{eq:cell-average}
\begin{equation}\label{eq:DGSEM_cell_averaged_2d}
 \langle u_h^{(n+1)}\rangle_\kappa = \langle u_h^{(n)}\rangle_\kappa - \frac{\Delta t^{(n)}}{|\kappa|}\sum_{e\in\partial\kappa}\sum_{k=0}^p\omega_kJ_e^k h\big(u_h^-({\bf x}_e^{k},t^{(n+1)}),u_h^+({\bf x}_e^{k},t^{(n+1)}),{\bf n}_e^k\big),
\end{equation}

\noindent providing that the discrete metric identities \cref{eq:discr_Met_Id} are satisfied \cite{kopriva_metric_id_06}.

We now characterize the properties of the DGSEM scheme \cref{eq:fully-discr_DGSEM} in \cref{th:ES_BE_DGSEM_noGV} for the unlimited scheme without artificial viscosity, and in \cref{th:exist_uniq_sol_BE_GV,th:ES_sol_BE_GV} for the scheme with graph viscosity. \Cref{th:ES_BE_DGSEM_noGV} is not a new result, but is recalled for the sake of comparison with \cref{th:exist_uniq_sol_BE_GV,th:ES_sol_BE_GV}.

\begin{lemma}\label{th:ES_BE_DGSEM_noGV}
The scheme \cref{eq:fully-discr_DGSEM} without artificial viscosity, $d_\kappa=0\;\forall\kappa\in\Omega_h$, and with the EC flux \cref{eq:entropy_conserv_flux} defined for a given entropy pair $\big(\vartheta(u),{\bf g}(u)\big)$, with $\vartheta''(\cdot)>0$, satisfies the following discrete entropy inequality for this pair:
\begin{align*}
 \langle \vartheta(u_h^{(n+1)})\rangle_\kappa \!\leq\! \langle \vartheta(u_h^{(n)})\rangle_\kappa \!-\! \frac{\Delta t^{(n)}}{|\kappa|}\!\!\!\sum_{e\in\partial\kappa}\sum_{k=0}^p\!\omega_kJ_e^k G_{ec}\big(u_h^-({\bf x}_e^{k},t^{(n+1)}),u_h^+({\bf x}_e^{k},t^{(n+1)}),{\bf n}_e^k\big).
\end{align*}
\noindent with $G_{ec}(u^-,u^+,{\bf n})=\tfrac{1}{2}h(u^-,u^+,{\bf n})(\vartheta'(u^-)+\vartheta'(u^+)) - \tfrac{1}{2}\big(\psib(u^-)+\psib(u^+)\big)\cdot{\bf n}$.
\end{lemma}

\begin{proof}
We multiply \cref{eq:fully-discr_DGSEM} with $\vartheta'(U_\kappa^{ij,n+1})$ and sum up over $0\leq i,j\leq p$. The space discretization with the EC flux \cref{eq:entropy_conserv_flux} is known to be ES, see, e.g., \cite{wintermeyer_etal_17}. Using convexity of the entropy, we get $\vartheta'(U_\kappa^{ij,n+1})(U_\kappa^{ij,n+1}-U_\kappa^{ij,n})\geq\vartheta(U_\kappa^{ij,n+1})-\vartheta(U_\kappa^{ij,n})$ which completes the proof.
\end{proof}

\begin{theorem}\label{th:exist_uniq_sol_BE_GV}
The scheme \cref{eq:fully-discr_DGSEM} with the EC flux \cref{eq:entropy_conserv_flux} defined for a given entropy pair $\big(\vartheta(u),{\bf g}(u)\big)$ with $\vartheta''(\cdot)>0$, and the graph viscosity in \cref{eq:graph_viscosity} such that
\begin{equation}\label{eq:condition_on_dkappa}
 d_\kappa \geq 4L_fL_u\max_{k\neq l}\frac{|D^p_{kl}|}{\omega_l} \quad \forall \kappa\in\Omega_h, \quad L_u := \sup_{m\leq u^+\neq u^-\leq M}\frac{\vartheta''(u^+)(u^+-u^-)}{v(u^+)-v(u^-)},
\end{equation}

\noindent with $v(u)=\vartheta'(u)$ and $L_f$ defined in \cref{eq:flux_lipschitz_cst}, has a unique solution satisfying the maximum principle: $m\leq U_\kappa^{ij,n+1}\leq M$ for all $\kappa\in\Omega_h$ and $0\leq i,j\leq p$.
\end{theorem}

\begin{proof}
\noindent\textit{Preliminaries.} We will use some short notations: for the sums, $\sum_\kappa=\sum_{\kappa\in\Omega_h}$, $\sum_k=\sum_{k=0}^p$, etc, also ${\bf f}_\kappa^{ij}={\bf f}(U_\kappa^{ij})$, $\eta_\kappa^{ij}=\eta(U_\kappa^{ij})$, etc. Let add the following trivial quantity to the LHS of \cref{eq:fully-discr_DGSEM}:
\begin{equation*}
 -2\omega_i\omega_jJ_\kappa^{ij}{\bf f}_\kappa^{ij,n+1}\cdot\Big(\sum_{k} D^p_{ik}{\bf n}_\kappa^{(ik)j} + D^p_{jk}{\bf n}_\kappa^{i(jk)}\Big) \overset{\cref{eq:discr_Met_Id}}{=} 0
\end{equation*}

\noindent with \cref{eq:semi-discr_DGSEM-res,eq:graph_viscosity} and rewrite the result multiplied by $\tfrac{\Delta t^{(n)}}{\omega_i\omega_jJ_\kappa^{ij}}$ as
\begin{align*}
U_\kappa^{ij,n+1} +& \frac{\Delta t^{(n)}}{J_\kappa^{ij}} \bigg( \\
&\sum_{k} 2D^p_{ik}\big(h_{ec}(U_\kappa^{ij},U_\kappa^{kj},{\bf n}_\kappa^{(ik)j})-{\bf f}_\kappa^{ij}\cdot{\bf n}_\kappa^{(ik)j}\big) + \frac{\omega_kd_\kappa|{\bf n}_\kappa^{(ik)j}|}{2}(U_\kappa^{ij}-U_\kappa^{kj}) \\
 +& \sum_{k} 2D^p_{jk}\big(h_{ec}(U_\kappa^{ij},U_\kappa^{ik},{\bf n}_\kappa^{i(jk)})-{\bf f}_\kappa^{ij}\cdot{\bf n}_\kappa^{i(jk)}\big) + \frac{\omega_kd_\kappa|{\bf n}_\kappa^{i(jk)}|}{2}(U_\kappa^{ij}-U_\kappa^{ik}) \\
 +& \frac{1}{\omega_i\omega_j} \sum_{e,k} \phi_\kappa^{ij}({\bf x}_e^k)\omega_kJ_e^k\Big(h\big(U_\kappa^{ij},u_h^{+}({\bf x}_e^{k},t),{\bf n}_e^k\big)-{\bf f}_\kappa^{ij}\cdot{\bf n}_e^k\Big)
 \bigg)^{(n+1)} = U_\kappa^{ij,n}
\end{align*}

\noindent\textit{Existence.} In the above scheme, we substitute
\begin{equation*}
 h\big(U_\kappa^{ij},u_h^{+}({\bf x}_e^{k},t),{\bf n}_e^k\big)-{\bf f}_\kappa^{ij}\cdot{\bf n}_e^k = 2L_fU_\kappa^{ij} -2L_f\Big(U_\kappa^{ij}-\frac{h\big(U_\kappa^{ij},u_h^{+}({\bf x}_e^{k},t),{\bf n}_e^k\big)-{\bf f}_\kappa^{ij}\cdot{\bf n}_e^k}{2L_f}\Big),
\end{equation*}

\noindent and
\begin{align*}
 2D^p_{ik}&\big(h_{ec}(U_\kappa^{ij},U_\kappa^{kj},{\bf n}_\kappa^{(ik)j})\!-\!{\bf f}_\kappa^{ij}\cdot{\bf n}_\kappa^{(ik)j}\big) + \frac{\omega_kd_\kappa|{\bf n}_\kappa^{(ik)j}|}{2}(U_\kappa^{ij} \!-\! U_\kappa^{kj}) = \frac{d_\kappa\omega_k|{\bf n}_\kappa^{(ik)j}|}{2\beta_{(ik)j}}U_\kappa^{ij} \\
  &- \frac{d_\kappa\omega_k|{\bf n}_\kappa^{(ik)j}|}{2\beta_{(ik)j}}\Big((1-\beta_{(ik)j})U_\kappa^{ij}+\beta_{(ik)j}U_\kappa^{kj} - \frac{h_{ec}(U_\kappa^{ij},U_\kappa^{kj},{\bf n}_\kappa^{(ik)j})-{\bf f}_\kappa^{ij}\cdot{\bf n}_\kappa^{(ik)j}}{\frac{d_\kappa\omega_k|{\bf n}_\kappa^{(ik)j}|}{4\beta_{(ik)j}D^p_{ik}}}\Big) \\
	=& \frac{d_\kappa\omega_k|{\bf n}_\kappa^{(ik)j}|}{2\beta_{(ik)j}}\Big(U_\kappa^{ij} - {\cal U}\big(U_\kappa^{ij},U_\kappa^{kj},\tfrac{{\bf n}_\kappa^{(ik)j}}{|{\bf n}_\kappa^{(ik)j}|}\big)\Big),
\end{align*}

\noindent with $\beta_{(ik)j}=\beta(U_\kappa^{ij},U_\kappa^{kj})$ defined in \cref{eq:def_beta_in_Uec}, ${\cal U}(\cdot,\cdot,\cdot)$ defined in \cref{eq:hec_over_fan-a} with $2\alpha=\tfrac{d_\kappa\omega_k}{4\beta_{(ik)j}D^p_{ik}}$ which satisfies $\alpha\geq L_f$ from \cref{eq:condition_on_dkappa}, $\beta_{(ik)j}\in(0,\tfrac{1}{2}]$, and $L_u\geq1$ in \cref{eq:condition_on_dkappa}. The latter inequality is due to $\lim_{u ^+\to u^-}v'(u^+)\tfrac{u^+-u^-}{v^+-v^-}=1$. Now we define $u_h^{(n+1,\mrm=0)}=u_h^{(n)}$ and introduce the following subiterations for $\mrm\geq0$
\begin{align}
\Big(1 +\sum_{k=0}^p(\gamma_\kappa^{(ik)j}&+\gamma_\kappa^{i(jk)})+\sum_{e\in\partial\kappa}\sum_{k=0}^p\gamma_e^k\Big)^{(n+1,\mrm)}U_\kappa^{ij,(n+1,\mrm+1)} = U_\kappa^{ij,n} \nonumber\\ 
+& \sum_{k=0}^p\Big(\gamma_\kappa^{(ik)j}{\cal U}(U_\kappa^{ij},U_\kappa^{kj},\tfrac{{\bf n}_\kappa^{(ik)j}}{|{\bf n}_\kappa^{(ik)j}|})+\gamma_\kappa^{i(jk)}{\cal U}(U_\kappa^{ij},U_\kappa^{ik},\tfrac{{\bf n}_\kappa^{i(jk)}}{|{\bf n}_\kappa^{i(jk)}|})\Big)^{(n+1,\mrm)} \nonumber\\ +& \sum_{e,k}\gamma_e^k \Big(U_\kappa^{ij}-\frac{h\big(U_\kappa^{ij},u_h^{+}({\bf x}_e^{k},t),{\bf n}_e^k\big)-{\bf f}_\kappa^{ij}\cdot{\bf n}_e^k}{2L_f}\Big)^{(n+1,\mrm)} \label{eq:def_Unp_mp_as_convex_comb}
\end{align}

\noindent so from \cref{eq:hec_over_fan-a,eq:MPP_interface_flux}, $U_\kappa^{ij,(n+1,\mrm+1)}$ is a convex combination of quantities in $[m,M]$ with coefficients
\begin{equation}\label{eq:conv_coeff_in_Un+1}
 \gamma_\kappa^{(ik)j}=\frac{\Delta t^{(n)}d_\kappa\omega_k|{\bf n}_\kappa^{(ik)j}|}{2\beta_{(ik)j}^{n+1,\mrm}J_\kappa^{ij}}, \gamma_\kappa^{i(jk)}=\frac{\Delta t^{(n)}d_\kappa\omega_k|{\bf n}_\kappa^{i(jk)}|}{2\beta_{i(jk)}^{n+1,\mrm}J_\kappa^{ij}}, \gamma_e^k = \frac{2\Delta t^{(n)}L_f\phi_\kappa^{ij}({\bf x}_e^k)\omega_kJ_e^k}{\omega_i\omega_jJ_\kappa^{ij}}.
\end{equation}

Since $u_h^{(n+1,\mrm+1)}$ lies in $[m,M]$, \cref{eq:def_Unp_mp_as_convex_comb} defines a map $u_h^{(n+1,\mrm)}\mapsto u_h^{(n+1,\mrm+1)}$ of the convex compact $[m,M]^{\mathbb{N}}$, subset of the normed vector space $\mathbb{R}^{\mathbb{N}}$, into itself. This map is continuous by Lipschitz continuity of the interface flux and regularity of the EC flux. By the Schauder fixed point theorem for infinite dimensional spaces, this map has at least one fixed point. The fixed point is solution to \cref{eq:fully-discr_DGSEM}.

\noindent\textit{Uniqueness.} We here adapt some existing tools \cite[\S~26]{EYMARD2000713} to our scheme. Let us assume two solutions $u_h^{(n+1)}$ and $w_h^{(n+1)}$ of \cref{eq:fully-discr_DGSEM} and drop the exponents $(n+1)$ for the sake of clarity, thus for all $\kappa\in\Omega_h$ and $0\leq i,j\leq p$, we have
\begin{equation*}
 \omega_i\omega_jJ_\kappa^{ij}\frac{U_\kappa^{ij}- W_\kappa^{ij}}{\Delta t^{(n)}} + R_\kappa^{ij}(u_h) + V_\kappa^{ij}(u_h) - R_\kappa^{ij}(w_h) - V_\kappa^{ij}(w_h) = 0.
\end{equation*}

Using the SBP property \cref{eq:SBP}, then symmetry and consistency of $h_{ec}$, we rewrite 
\begin{align*}
2Q^p_{ik}h_{ec}&(U_\kappa^{ij},U_\kappa^{kj},{\bf n}_\kappa^{(ik)j}) = (Q^p_{ik}-Q^p_{ki})h_{ec}(U_\kappa^{ij},U_\kappa^{kj},{\bf n}_\kappa^{(ik)j}) +\delta_{ip}\delta_{kp}{\bf f}_\kappa^{pj}\cdot{\bf n}_\kappa^{(pp)j} \\ &-\delta_{i0}\delta_{k0}{\bf f}_\kappa^{0j}\cdot{\bf n}_\kappa^{(00)j} \\
=& \frac{Q^p_{ik}-Q^p_{ki}}{2}h_{ec}(U_\kappa^{ij},U_\kappa^{kj},{\bf n}_\kappa^{(ik)j})-\frac{Q^p_{ki}-Q^p_{ik}}{2}h_{ec}(U_\kappa^{kj},U_\kappa^{ij},{\bf n}_\kappa^{(ki)j}) \\
&+\delta_{ip}\delta_{kp}{\bf f}_\kappa^{pj}\cdot{\bf n}_e({\bf x}_\kappa^{pj})+\delta_{i0}\delta_{k0}{\bf f}_\kappa^{0j}\cdot{\bf n}_e({\bf x}_\kappa^{0j}),
\end{align*}

\noindent where we have used \cref{eq:link_vol_surf_metric} in the last step. We thus recast the DGSEM residuals \cref{eq:semi-discr_DGSEM-res} plus \cref{eq:graph_viscosity} as

\begin{align*}
 R_\kappa^{ij}(u_h) + V_\kappa^{ij}(u_h)  &= \frac{\omega_j}{2}\sum_{k} G_\kappa^{(ik)j}(U_\kappa^{ij},U_\kappa^{kj})-G_\kappa^{(ik)j}(U_\kappa^{kj},U_\kappa^{ij}) \\
 &+ \frac{\omega_i}{2}\sum_{k} G_\kappa^{i(jk)}(U_\kappa^{ij},U_\kappa^{ik})-G_\kappa^{i(jk)}(U_\kappa^{ik},U_\kappa^{ij}) \\
 &+ \sum_{e,k} \phi_\kappa^{ij}({\bf x}_e^k)\omega_kJ_e^k h\big(U_\kappa^{ij},u_h^{+}({\bf x}_e^{k},t),{\bf n}_e^k\big),
\end{align*}

\noindent where 
\begin{align*}
 G_\kappa^{(ik)j}(U_\kappa^{ij},U_\kappa^{kj}) &= (Q^p_{ik}-Q^p_{ki})h_{ec}\big(U_\kappa^{ij},U_\kappa^{kj},{\bf n}_\kappa^{(ik)j}\big) + d_\kappa \omega_i\omega_k|{\bf n}_\kappa^{(ik)j}| (U_\kappa^{ij}-U_\kappa^{kj}) \\
 G_\kappa^{(ik)j}(U_\kappa^{kj},U_\kappa^{ij}) &= (Q^p_{ki}-Q^p_{ik})h_{ec}\big(U_\kappa^{kj},U_\kappa^{ij},{\bf n}_\kappa^{(ik)j}\big) + d_\kappa \omega_k\omega_j|{\bf n}_\kappa^{(ik)j}| (U_\kappa^{kj}-U_\kappa^{ij}) \\
 G_\kappa^{i(jk)}(U_\kappa^{ij},U_\kappa^{ik}) &= (Q^p_{jk}-Q^p_{kj})h_{ec}\big(U_\kappa^{ij},U_\kappa^{ik},{\bf n}_\kappa^{i(jk)}\big) + d_\kappa \omega_j\omega_k|{\bf n}_\kappa^{i(jk)}| (U_\kappa^{ij}-U_\kappa^{ik}) \\
 G_\kappa^{i(jk)}(U_\kappa^{ik},U_\kappa^{ij}) &= (Q^p_{kj}-Q^p_{jk})h_{ec}\big(U_\kappa^{ik},U_\kappa^{ij},{\bf n}_\kappa^{i(jk)}\big) + d_\kappa \omega_k\omega_j|{\bf n}_\kappa^{i(jk)}| (U_\kappa^{ik}-U_\kappa^{ij})
\end{align*}

\noindent are skew-symmetric, e.g., $G_\kappa^{(ik)j}(U_\kappa^{ij},U_\kappa^{kj})=-G_\kappa^{(ik)j}(U_\kappa^{kj},U_\kappa^{ij})$, and are nondecreasing functions of their first argument under \cref{eq:condition_on_dkappa}. Indeed, let write $G(u)=(Q^p_{ik}-Q^p_{ki})h_{ec}\big(u,u^+,{\bf n}\big) + d_\kappa \omega_i\omega_k|{\bf n}|(u-u^+)$, then $G'(u)=(Q^p_{ik}-Q^p_{ki})\partial_1h_{ec}(u,u^+,{\bf n}) + d_\kappa \omega_i\omega_k|{\bf n}|$. Using \cref{eq:entropy_conserv_flux} and $\psib'(u)=v'(u){\bf f}(u)^\top$, then $\psib\circ u(v)=\int^v{\bf f}\circ u(w)dw$ we obtain
\begin{align}
 |\partial_1h_{ec}(u,u^+,{\bf n})| &= \bigg|\frac{v'(u)\Big(\psib(u^+)-\psib(u)-(v^+-v){\bf f}(u)\Big)\cdot{\bf n}}{(v^+-v)^2}\bigg| \nonumber\\ 
 &= \frac{v'(u)}{(v^+-v)^2}\Big|(v^+-v)\Big(\int_0^1{\bf f}\circ u\big(\theta v^++(1-\theta)v\big)d\theta-{\bf f}(u)\Big)\cdot{\bf n}\Big| \nonumber\\ 
 &\leq \frac{v'(u)}{|v^+-v|}L_f|u^+-u|\,|{\bf n}| \leq L_fL_u|{\bf n}|, \label{eq:bound_partial_hec}
\end{align}

\noindent from the definition of $L_u$ in \cref{eq:condition_on_dkappa}. By monotonicity, $h(\cdot,\cdot,{\bf n})$ is also nondecreasing
with its first argument and we have
\begin{align*}
 \omega_i\omega_jJ_\kappa^{ij}\frac{|U_\kappa^{ij}- W_\kappa^{ij}|}{\Delta t^{(n)}} &+ \frac{\omega_j}{2}\sum_{k} \big|G_\kappa^{(ik)j}(U_\kappa^{ij},U_\kappa^{kj})-G_\kappa^{(ik)j}(W_\kappa^{ij},U_\kappa^{kj})\big| \\
 &+ \frac{\omega_j}{2}\sum_{k} \big|G_\kappa^{(ik)j}(U_\kappa^{ij},W_\kappa^{kj})-G_\kappa^{(ik)j}(W_\kappa^{ij},W_\kappa^{kj})\big| \\
 &+ \frac{\omega_i}{2}\sum_{k} \big|G_\kappa^{i(jk)}(U_\kappa^{ij},U_\kappa^{ik})-G_\kappa^{i(jk)}(W_\kappa^{ij},U_\kappa^{ik})\big| \\
 &+ \frac{\omega_i}{2}\sum_{k} \big|G_\kappa^{i(jk)}(U_\kappa^{ij},W_\kappa^{ik})-G_\kappa^{i(jk)}(W_\kappa^{ij},W_\kappa^{ik})\big| \\
 &+ \sum_{e,k} \frac{\phi_\kappa^{ij}({\bf x}_e^k)\omega_kJ_e^k}{2} \big|h(U_\kappa^{ij},u_h^{+},{\bf n}_e^k)-h(W_\kappa^{ij},u_h^{+},{\bf n}_e^k)\big|_{{\bf x}_e^{k}} \\
 &+ \sum_{e,k} \frac{\phi_\kappa^{ij}({\bf x}_e^k)\omega_kJ_e^k}{2} \big|h(U_\kappa^{ij},w_h^{+},{\bf n}_e^k)-h(W_\kappa^{ij},w_h^{+}{\bf n}_e^k)\big|_{{\bf x}_e^{k}} \\
 &\leq \frac{\omega_j}{2}\sum_{k} \big|G_\kappa^{(ik)j}(W_\kappa^{ij},U_\kappa^{kj})-G_\kappa^{(ik)j}(W_\kappa^{ij},W_\kappa^{kj})\big| \\
 &+ \frac{\omega_j}{2}\sum_{k} \big|G_\kappa^{(ik)j}(U_\kappa^{ij},U_\kappa^{kj})-G_\kappa^{(ik)j}(U_\kappa^{ij},W_\kappa^{kj})\big| \\
 &+ \frac{\omega_i}{2}\sum_{k} \big|G_\kappa^{i(jk)}(W_\kappa^{ij},U_\kappa^{ik})-G_\kappa^{i(jk)}(W_\kappa^{ij},W_\kappa^{ik})\big| \\
 &+ \frac{\omega_i}{2}\sum_{k} \big|G_\kappa^{i(jk)}(U_\kappa^{ij},U_\kappa^{ik})-G_\kappa^{i(jk)}(U_\kappa^{ij},W_\kappa^{ik})\big| \\
 &+ \sum_{e,k} \frac{\phi_\kappa^{ij}({\bf x}_e^k)\omega_kJ_e^k}{2} \big|h(U_\kappa^{ij},u_h^{+},{\bf n}_e^k)-h(U_\kappa^{ij},w_h^{+},{\bf n}_e^k)\big|_{{\bf x}_e^{k}} \\
 &+ \sum_{e,k} \frac{\phi_\kappa^{ij}({\bf x}_e^k)\omega_kJ_e^k}{2} \big|h(W_\kappa^{ij},u_h^{+},{\bf n}_e^k)-h(W_\kappa^{ij},w_h^{+}{\bf n}_e^k)\big|_{{\bf x}_e^{k}} 
\end{align*}

\noindent since all the terms between vertical bars in the LHS are of the same sign. Multiplying the above relation by a quantity $\varphi_\kappa^{ij}\geq0$, which will be defined below, and summing over $\kappa\in\Omega_h$ and $0\leq i,j\leq p$, then switching indices $i\leftrightarrow k$ and $j\leftrightarrow k$ in the sums over cells $\kappa$ and switching the left and right states in the sums over faces $e$ we get
\begin{align}
 \sum_{\kappa,i,j}a_\kappa^{ij}|U_\kappa^{ij}-W_\kappa^{ij}|
 &\leq \sum_{\kappa}\sum_{i,j,k} \frac{\omega_j}{2}\big|G_\kappa^{(ik)j}(W_\kappa^{ij},U_\kappa^{kj})-G_\kappa^{(ik)j}(W_\kappa^{ij},W_\kappa^{kj})\big||\varphi_\kappa^{ij}-\varphi_\kappa^{kj}|\nonumber\\
 +& \sum_{\kappa}\sum_{i,j,k} \frac{\omega_j}{2}\big|G_\kappa^{(ik)j}(U_\kappa^{ij},U_\kappa^{kj})-G_\kappa^{(ik)j}(U_\kappa^{ij},W_\kappa^{kj})\big||\varphi_\kappa^{ij}-\varphi_\kappa^{kj}|\nonumber\\
 +& \sum_{\kappa}\sum_{i,j,k} \frac{\omega_i}{2}\big|G_\kappa^{i(jk)}(W_\kappa^{ij},U_\kappa^{ik})-G_\kappa^{i(jk)}(W_\kappa^{ij},W_\kappa^{ik})\big||\varphi_\kappa^{ij}-\varphi_\kappa^{ik}|\nonumber\\
 +& \sum_{\kappa}\sum_{i,j,k} \frac{\omega_i}{2}\big|G_\kappa^{i(jk)}(U_\kappa^{ij},U_\kappa^{ik})-G_\kappa^{i(jk)}(U_\kappa^{ij},W_\kappa^{ik})\big||\varphi_\kappa^{ij}-\varphi_\kappa^{ik}|\nonumber\\
 +& \sum_{\kappa}\sum_{e,k} \frac{\omega_kJ_e^k}{2} \Big(\big|h(u_h^{-},u_h^{+},{\bf n}_e^k)-h(w_h^{-},u_h^{+},{\bf n}_e^k)\big|\nonumber\\ &\qquad+\big|h(u_h^{-},w_h^{+},{\bf n}_e^k)-h(w_h^{-},w_h^{+},{\bf n}_e^k)\big|\Big)_{{\bf x}_e^{k}}|\varphi_h^+-\varphi_h^-|_{{\bf x}_e^{k}}\nonumber\\
 \leq& \sum_{\kappa}\sum_{i,j,k} \omega_j|{\bf n}_\kappa^{(ik)j}|\big(|Q^p_{ik}\!-\!Q^p_{ki}|L_fL_u\!+\!d_\kappa\omega_i\omega_k\big)|U_\kappa^{kj}\!-\!W_\kappa^{kj}||\varphi_\kappa^{ij}\!-\!\varphi_\kappa^{kj}|\nonumber\\
 +& \sum_{\kappa}\sum_{i,j,k} \omega_i|{\bf n}_\kappa^{i(jk)}|\big(|Q^p_{jk}\!-\!Q^p_{kj}|L_fL_u\!+\!d_\kappa\omega_k\omega_j\big)|U_\kappa^{ik}\!-\!W_\kappa^{ik}||\varphi_\kappa^{ij}\!-\!\varphi_\kappa^{ik}|\nonumber\\
 +& \sum_{\kappa}\sum_{e,k} \frac{\omega_kJ_e^k}{2} (L_1+L_2)|u_h^--w_h^-|_{{\bf x}_e^{k}}|\varphi_h^+-\varphi_h^-|_{{\bf x}_e^{k}}\nonumber\\
 =& \sum_{\kappa,i,j} b_\kappa^{ij} |U_\kappa^{ij}-W_\kappa^{ij}| \label{eq:ineq_with_a-b_coeffs}
\end{align}

\noindent where we have used symmetry of $h_{ec}(\cdot,\cdot,{\bf n})$ and the mean value theorem to get
\begin{align*}
 |h_{ec}(u,u^+,{\bf n})&-h_{ec}(u,u^-,{\bf n})| =|h_{ec}(u^+,u,{\bf n})-h_{ec}(u^-,u,{\bf n})| \\ =& \big|\partial_1h_{ec}\big(u,\theta u^++(1-\theta)u^-,{\bf n}\big)\big|\;|u^+-u^-| \overset{\cref{eq:bound_partial_hec}}{\leq} L_fL_u|{\bf n}|\;|u^+-u^-|,
\end{align*}

\noindent while $a_\kappa^{ij}=\tfrac{\omega_i\omega_jJ_\kappa^{ij}}{\Delta t^{(n)}}\varphi_\kappa^{ij}$, and from \cref{eq:link_vol_surf_metric} we obtain
\begin{align*}
 b_\kappa^{ij} &= \sum_k \omega_j|{\bf n}_\kappa^{(ik)j}|\big(|Q^p_{ik}-Q^p_{ki}|L_fL_u+d_\kappa\omega_i\omega_k\big)|\varphi_\kappa^{ij}-\varphi_\kappa^{kj}| \\
 &+ \sum_k \omega_i|{\bf n}_\kappa^{i(jk)}|\big(|Q^p_{jk}-Q^p_{kj}|L_fL_u+d_\kappa\omega_k\omega_j\big)|\varphi_\kappa^{ij}-\varphi_\kappa^{ik}| \\
 +& \tfrac{L_1+L_2}{2}\big(\omega_j(\delta_{i0}|{\bf n}_\kappa^{(00)j}|+\delta_{ip}|{\bf n}_\kappa^{(pp)j}|) + \omega_i(\delta_{j0}|{\bf n}_\kappa^{i(00)}|+\delta_{jp}|{\bf n}_\kappa^{i(pp)}|)\big)|\varphi_\kappa^{ij}-\varphi_h^+({\bf x}_\kappa^{ij})|.
\end{align*}


We now define $\varphi_h$ as the space average of $\varphi({\bf x})=e^{-\gamma|{\bf x}|}$ in subcells $\cup_{i,j=0}^p\ol{\kappa_{ij}}=\ol{\kappa}$: $\varphi_h({\bf x}')=\tfrac{1}{|\kappa_{ij}|}\int_{\kappa_{ij}}\varphi({\bf x})dV$ $\forall{\bf x}'\in\kappa_{ij}$. Such subcells $\kappa_{ij}$ of volume $|\kappa_{ij}|=\omega_i\omega_jJ_\kappa^{ij}$ have been shown to exist on high-order grids \cite[App.~B.3]{HENNEMANN_etal_FVGDSEM_21}. We thus have $\sum_\kappa\langle\varphi_h\rangle_\kappa\leq\int_{\mathbb{R}^d}e^{-\gamma|{\bf x}|}dV$ finite providing $\gamma>0$.

Using \cref{eq:regular_mesh_cond}  and $|Q^p_{ik}-Q^p_{ki}|L_fL_u\leq\omega_k\omega_id_\kappa$ from \cref{eq:condition_on_dkappa}, we have 
\begin{equation}\label{eq:bounds_on_a-b_coeffs}
 a_\kappa^{ij} \geq \frac{\omega_0^d\beta h^d}{2d\Delta t^{(n)}}\inf_{{\bf x}\in B_\kappa(h)}\varphi({\bf x}), \quad b_\kappa^{ij}\leq \frac{(8d_\kappa+L_1+L_2)h^d}{d\beta}\sup_{{\bf x}\in B_\kappa(2h)}|\nabla\varphi|,
\end{equation}

\noindent with $B_\kappa(h)$ the ball of radius $h$ centered on the barycenter of $\kappa$. Therefore choosing $\gamma>0$ such that $\gamma e^{3\gamma h}<\tfrac{\omega_0^d\beta^2}{2(8d_\kappa+L_1+L_2)\Delta t^{(n)}}$ imposes $a_\kappa^{ij}>b_\kappa^{ij}$, so inequality \cref{eq:ineq_with_a-b_coeffs} cannot hold unless $U_\kappa^{ij}=W_\kappa^{ij}$ for all the DOFs, which completes the proof.
\end{proof}

As a direct consequence of \cref{eq:def_Unp_mp_as_convex_comb} with the initial data $u_h^{(n+1,\mrm=0)}=u_h^{(n)}$ and uniqueness of the solution to \cref{eq:fully-discr_DGSEM}, we obtain the following discrete invariance property.

\begin{corollary}\label{th:corollary_MP_BE_GV}
 Under the assumptions of \cref{th:exist_uniq_sol_BE_GV}, the solution to \cref{eq:fully-discr_DGSEM} with the artificial viscosity in \cref{eq:graph_viscosity,eq:condition_on_dkappa} satisfies the estimate
\begin{equation*}
 \min_{\kappa'\in\Omega_h}\min_{0\leq k,l\leq p} U_{\kappa'}^{kl,n} \leq U_\kappa^{ij,n+1} \leq \max_{\kappa'\in\Omega_h}\max_{0\leq k,l\leq p} U_{\kappa'}^{kl,n} \quad\forall \kappa\in\Omega_h, 0\leq i,j\leq p.
\end{equation*}
\end{corollary}


%
\begin{theorem}\label{th:ES_sol_BE_GV}
The scheme \cref{eq:fully-discr_DGSEM} with the EC flux \cref{eq:entropy_conserv_flux}, defined for a given entropy pair $\big(\vartheta(u),{\bf g}(u)\big)$ with $\vartheta''(\cdot)>0$, and the artificial viscosity in \cref{eq:graph_viscosity,eq:condition_on_dkappa} satisfies the following inequality for every entropy pair $\big(\eta(u),{\bf q}(u)\big)$ with $\eta''(\cdot)\geq0$:

\begin{equation*}
 \langle\eta(u_h^{(n+1)})\rangle_\kappa \leq \langle\eta(u_h^{(n)})\rangle_\kappa - \frac{\Delta t^{(n)}}{|\kappa|}\sum_{e\in\partial\kappa}\sum_{k=0}^p\omega_kJ_e^k Q\big(u_h^-({\bf x}_e^{k},t^{(n+1)}),u_h^+({\bf x}_e^{k},t^{(n+1)}),{\bf n}_e^k\big),
\end{equation*}

\noindent with the numerical entropy flux $Q(\cdot,\cdot,\cdot)$ defined in \cref{eq:ES_interface_flux}.
\end{theorem}

\begin{proof}
We again use short notations as $\eta_\kappa^{ij}=\eta(U_\kappa^{ij})$ and ${\bf q}_\kappa^{ij}={\bf q}(U_\kappa^{ij})$, and drop the time exponent $(n+1)$ for the sake of clarity. We know that $u_h^{(n+1,\mrm+1)}=u_h^{(n+1,\mrm)}=u_h^{(n+1)}$ is solution to \cref{eq:def_Unp_mp_as_convex_comb}. Applying the entropy $\eta(\cdot)$ to the convex combination in \cref{eq:def_Unp_mp_as_convex_comb} and using \cref{eq:ES_interface_flux,eq:hec_over_fan-b}, we get
\begin{align*}
 \Big(1 +&\sum_{k}(\gamma_\kappa^{(ik)j}+\gamma_\kappa^{i(jk)})+\sum_{e,k}\gamma_e^k\Big)\eta_\kappa^{ij} \leq \eta_\kappa^{ij,n} \nonumber\\ 
&+ \sum_{k}\gamma_\kappa^{(ik)j}\bigg(\frac{\ol\eta(U_\kappa^{ij},U_\kappa^{kj})+\eta_\kappa^{ij}}{2}-\frac{q_{ec}(U_\kappa^{ij},U_\kappa^{kj},{\bf n}_\kappa^{(ik)j})-{\bf q}_\kappa^{ij}\cdot{\bf n}_\kappa^{(ik)j}}{\tfrac{d_\kappa\omega_k|{\bf n}_\kappa^{(ik)j}|}{4\beta_{(ik)j}D^p_{ik}}}\bigg) \\
&+ \sum_{k}\gamma_\kappa^{i(jk)}\bigg(\frac{\ol\eta(U_\kappa^{ij},U_\kappa^{ik})+\eta_\kappa^{ij}}{2}-\frac{q_{ec}(U_\kappa^{ij},U_\kappa^{ik},{\bf n}_\kappa^{i(jk)})-{\bf q}_\kappa^{ij}\cdot{\bf n}_\kappa^{i(jk)}}{\tfrac{d_\kappa\omega_k|{\bf n}_\kappa^{i(jk)}|}{4\beta_{i(jk)}D^p_{jk}}}\bigg) \\
&+ \sum_{e,k}\gamma_e^k \Big(\eta_\kappa^{ij}-\frac{Q\big(U_\kappa^{ij},u_h^{+}({\bf x}_e^{k},t),{\bf n}_e^k\big)-{\bf q}_\kappa^{ij}\cdot{\bf n}_e^k}{2L_f}\Big),
\end{align*}

\noindent with $\ol{\eta}$  defined in \cref{eq:def_ol_eta}. Using the metric identities \cref{eq:discr_Met_Id}, the terms in ${\bf q}_\kappa^{ij}$ in the second and third lines vanish, then using the expressions of the convex coefficients \cref{eq:conv_coeff_in_Un+1} and multiplying the result by $\omega_i\omega_jJ_\kappa^{ij}/\Delta t^{(n)}$ gives
\begin{align*}
 \omega_i\omega_j&J_\kappa^{ij}\frac{\eta_\kappa^{ij}-\eta_\kappa^{ij,n}}{\Delta t^{(n)}} + 2\sum_{k}\omega_jQ^p_{ik}q_{ec}(U_\kappa^{ij},U_\kappa^{kj},{\bf n}_\kappa^{(ik)j}) + \omega_iQ^p_{jk}q_{ec}(U_\kappa^{ij},U_\kappa^{ik},{\bf n}_\kappa^{i(jk)}) \\
 &+ \sum_{e,k} \phi_\kappa^{ij}({\bf x}_e^k)\omega_kJ_e^k\Big(Q\big(U_\kappa^{ij},u_h^{+}({\bf x}_e^{k},t),{\bf n}_e^k\big)-{\bf q}_\kappa^{ij}\cdot{\bf n}_e^k\Big) \\
 \leq &d_\kappa\omega_i\omega_j\sum_{k}\frac{\omega_k}{2}\Big(\frac{|{\bf n}_\kappa^{(ik)j}|}{\beta_{(ik)j}}\big(\ol\eta(U_\kappa^{ij},U_\kappa^{kj})-\eta_\kappa^{ij}\big)+\frac{|{\bf n}_\kappa^{i(jk)}|}{\beta_{i(jk)}}\big(\ol\eta(U_\kappa^{ij},U_\kappa^{ik})-\eta_\kappa^{ij}\big)\Big) \\
 \overset{\cref{eq:ineq_oleta}}{\leq} &d_\kappa\omega_i\omega_j\sum_{k}\omega_k\Big(|{\bf n}_\kappa^{(ik)j}|(\eta_\kappa^{kj}-\eta_\kappa^{ij})+|{\bf n}_\kappa^{i(jk)}|(\eta_\kappa^{ik}-\eta_\kappa^{ij})\Big).
\end{align*}

The LHS corresponds to a time implicit DGSEM discretization of \cref{eq:PDE_entropy_ineq}. By conservation of the DGSEM, summing the above inequality over $0\leq i,j\leq p$, the $q_{ec}(\cdot,\cdot,\cdot)$ terms in the first row of the LHS cancel out with the ${\bf q}_\kappa^{ij}$ terms in the second row from symmetry of $q_{ec}(\cdot,\cdot,{\bf n})$ in \cref{eq:def_qec} and the SBP property \cref{eq:SBP}. The RHS vanishes since $\sum_{i,j,k}\omega_i\omega_j\omega_k|{\bf n}_\kappa^{(ik)j}|(\eta_\kappa^{kj}-\eta_\kappa^{ij})=0$ by switching indices $i\leftrightarrow k$ and using symmetry of ${\bf n}_\kappa^{(ik)j}={\bf n}_\kappa^{(ki)j}$ in \cref{eq:def_Jkappa_mean}. The same holds for the second sum in the RHS which concludes the proof.
\end{proof}

\begin{remark}
 Properties of \cref{th:exist_uniq_sol_BE_GV,th:ES_sol_BE_GV} also hold without EC fluxes, i.e., with $h_{ec}(u^-,u^+,{\bf n})=\tfrac{1}{2}\big({\bf f}(u^-)+{\bf f}(u^+)\big)\cdot{\bf n}$ in \cref{eq:fully-discr_DGSEM,eq:semi-discr_DGSEM-res}, but we look for graph viscosity \cref{eq:graph_viscosity} as a minimal modification of the scheme to make it MPP and stable and want to keep the scheme as close as it is used in practice. The modification \cref{eq:graph_viscosity} to the DGSEM scheme is indeed local to the cell, easy to implement and to differentiate. Likewise, $d_\kappa$ can be computed a priori. In \cref{sec:num_xp}, we propose to locally tune the graph viscosity as a function of the local regularity of the solution, which requires to use the ES DGSEM when the graph viscosity is removed by the local tuning.
\end{remark}

\begin{remark}[square entropy]\label{rk:square_entropy2}
Conditions \cref{eq:condition_on_dkappa} may be relaxed for the square entropy, $\vartheta(u)=\tfrac{u^2}{2}$ (see \cref{rk:square_entropy1}), since $L_u=1$, $\beta(u^-,u^+)=\tfrac{1}{4}$, and $|h_{ec}(u^+,u,{\bf n})-h_{ec}(u^-,u,{\bf n})|\leq L_f\int_0^1\theta d\theta|u^+-u^-|$, so $d_\kappa \geq 2L_f\max_{k\neq l}\frac{|D^p_{kl}|}{\omega_k}$ is enough.
\end{remark}

%
%
\section{Space-time DGSEM discretization}\label{sec:space_time_DGSEM_2D}

We now look for a numerical solution to \cref{eq:hyp_cons_laws} in the function space ${\cal V}_h^{p,q}$ of piecewise polynomials of degrees $p\geq1$ in space (here $d=2$) and $q\geq1$ in time:
\begin{equation*}
 u_h({\bf x},t) = \sum_{i,j=0}^p\sum_{r=0}^{q}\phi_\kappa^{ij}({\bf x})\phi_n^r(t) U_\kappa^{ij}(t_n^r), \quad \forall ({\bf x},t) \in\kappa\times(t^{(n)},t^{(n+1)}),
\end{equation*}

\noindent with $\phi_n^r(t)=\ell_r\big(2\tfrac{t-t^{(n)}}{\Delta t^{(n)}}-1\big)$, with again the Gauss-Lobatto quadrature rules and associated Lagrange polynomials as basis functions in both space and time. Hence, every DOF $U_\kappa^{ij}(t_n^r)$ corresponds to $u_h({\bf x}_\kappa^{ij},t_n^r)$ with $t_n^r=t^{(n)}+\tfrac{\Delta t^{(n)}}{2}(1+\xi_r)$, so $t_n^0=t^{(n)^+}$ and $t_n^q=t^{(n+1)^-}$. The space-time DGSEM reads: find $u_h\in{\cal V}_h^{p,q}$ such that for all $\kappa\in\Omega_h$, $0\leq i,j\leq p$ and $0\leq r\leq q$ we have
\begin{equation}\label{eq:space_time_DGSEM}
  \omega_i\omega_jJ_\kappa^{ij}T_\kappa^r\big(u_h({\bf x}_\kappa^{ij},\cdot)\big) + \frac{\omega_r \Delta t^{(n)}}{2}\Big(R_\kappa^{ij}\big(u_h(\cdot,t_n^r)\big)+V_\kappa^{ij}\big(u_h(\cdot,t_n^r)\big)\Big)=0, 
\end{equation}

\noindent together with the initial condition $U_\kappa^{ij}(t_{-1}^q)=u_0(x_\kappa^{ij})$. With some slight abuse, we remove the dependence in $p$ and $q$ of the quadrature weights, but indices $r,m$ (resp., $i,j,k,l$) will always refer to indices in the range $\du0,q\df$ (resp., $\du0, p\df$). The space discretization is defined in \cref{eq:semi-discr_DGSEM-res,eq:graph_viscosity}, while the discretization of $\partial_tu$ (mutliplied by $\Delta t^{(n)}$) reads
\begin{align}\label{eq:stDGSEM_time_discr}
 T_\kappa^r\big(u_h({\bf x}_\kappa^{ij},\cdot)\big) =& \sum_{m=0}^q2Q^q_{rm}U_{ec}\big(U_\kappa^{ij}(t_n^r),U_\kappa^{ij}(t_n^m)\big) + \delta_{r0}\big(U_\kappa^{ij}(t_n^0)-U_\kappa^{ij}(t_{n-1}^q)\big) \nonumber\\ &+ d_n\omega_r \sum_{m=0}^q \omega_m\big(U_\kappa^{ij}(t_n^r)-U_\kappa^{ij}(t_n^m)\big),
\end{align}

\noindent where we have used an upwind numerical flux in time to satisfy causality, while the last term is again a graph viscosity term and
\begin{subequations}\label{eq:EC_time_flux}
\begin{equation}\label{eq:EC_time_flux-a}
 U_{ec}(u^-,u^+) \!=\! \int_0^1u\big(\theta v(u^+)+(1-\theta)v(u^-)\big)d\theta \overset{\cref{eq:def_beta_in_Uec}}{=} 2\beta(u^-,u^+)(u^+-u^-)+u^-,
\end{equation}

\noindent with $v(u)=\vartheta'(u)$, is smooth, consistent, $U_{ec}(u,u)=u$, and EC in the sense \cite{lefloch2002fully,friedrich_etal_sp_time_DG_2019} 
\begin{equation}\label{eq:EC_time_flux-b}
 \big(v(u^+)-v(u^-)\big)U_{ec}(u^-,u^+) = \psi_t(u^+)-\psi_t(u^-), \quad \psi_t(u)=v(u)u-\vartheta(u),
\end{equation}
\end{subequations}

\noindent with $\psi_t(u)$ the entropy potential. By \cref{eq:EC_time_flux-b}, we have $\partial_1U_{ec}(u^-,u^+)=\tfrac{v'(u^-)}{(v^+-v^-)^2}\big(\vartheta^--\vartheta^+-v^+(u^--u^+)\big)\in[0,L_u]$ by convexity of $\vartheta(u)$, $v^-(u^--u^+)\geq\vartheta^--\vartheta^+\geq v^+(u^--u^+)$, and the definition of $L_u$ in \cref{eq:condition_on_dkappa}. By symmetry, $\partial_2U_{ec}(u^-,u^+)=\partial_1U_{ec}(u^+,u^-)$.

Summing \cref{eq:space_time_DGSEM} over $0\leq r\leq q$ and $0\leq i,j\leq p$, we obtain
\begin{equation*}
 \avg{u_h(t_n^q)}_\kappa = \avg{u_h(t_{n-1}^q)}_\kappa - \frac{\Delta t^{(n)}}{|\kappa|}\sum_{r=0}^q \frac{\omega_r}{2}\sum_{e\in\partial\kappa}\sum_{k=0}^p\omega_kJ_e^k h\big(u_h^-({\bf x}_e^{k},t_n^r),u_h^+({\bf x}_e^{k},t_n^r),{\bf n}_e^k\big)
\end{equation*}

\noindent which may be compared to \cref{eq:DGSEM_cell_averaged_2d}. The following lemma has been established in \cite{friedrich_etal_sp_time_DG_2019} and is given here for the sake of comparison with \cref{th:ES_sol_st_DGSEM}.

\begin{lemma}\label{th:ES_stDGSEM_noGV}
The space-time DGSEM \cref{eq:space_time_DGSEM} without artificial viscosity, $d_\kappa=d_n=0$ $\forall\kappa\in\Omega_h$ and $n\geq0$, and with the EC fluxes \cref{eq:entropy_conserv_flux,eq:EC_time_flux} defined for a given entropy pair $\big(\vartheta(u),{\bf g}(u)\big)$, with $\vartheta''(\cdot)>0$, satisfies the following discrete entropy inequality for this pair:
\begin{align*}
 \langle \vartheta\big(u_h(t_{n}^q)\big)\rangle_\kappa &\leq \langle \vartheta\big(u_h(t_{n-1}^q)\big)\rangle_\kappa  \\ &\!-\! \frac{\Delta t^{(n)}}{|\kappa|}\!\!\sum_{r=0}^q \frac{\omega_r}{2}\sum_{e\in\partial\kappa}\sum_{k=0}^p\omega_kJ_e^k G_{ec}\big(u_h^-({\bf x}_e^{k},t_n^r),u_h^+({\bf x}_e^{k},t_n^r),{\bf n}_e^k\big).
\end{align*}

\noindent with $G_{ec}(u^-,u^+,{\bf n})=\tfrac{1}{2}h(u^-,u^+,{\bf n})(\vartheta'(u^-)+\vartheta'(u^+)) - \tfrac{1}{2}\big(\psib(u^-)+\psib(u^+)\big)\cdot{\bf n}$.
\end{lemma}

\begin{proof}
We here only recall the main steps of the proof and refer to \cite[Th.~1 \& 2]{friedrich_etal_sp_time_DG_2019} for details. We set $d_\kappa=d_n=0$ and multiply \cref{eq:space_time_DGSEM} with $V_\kappa^{ij}(t_n^r):=\vartheta'\big(U_\kappa^{ij}(t_n^r)\big)$ and sum up over $0\leq i,j\leq p$ and $0\leq r\leq q$. The space discretization with the EC flux \cref{eq:entropy_conserv_flux} is known to be ES from \cref{th:ES_sol_BE_GV}. Since $T_\kappa^r$ is only evaluated at $U_\kappa^{ij}(\cdot)$ in space, we here remove the space dependency and use the notations like $U_n^r=U_\kappa^{ij}(t_n^r)$ to get
\begin{align*}
 \sum_{r=0}^q V_n^rT_\kappa^r \overset{\cref{eq:SBP}}{=}& \sum_{r,m}Q^q_{rm}V_n^rU_{ec}\big(U_n^r,U_n^m\big)-Q^q_{mr}V_n^rU_{ec}\big(U_n^r,U_n^m\big) \\ &+ V_n^qU_n^q-V_n^0U_n^0+V_n^0\big(U_n^0-U_{n-1}^q\big)  \\
 \overset{\cref{eq:EC_time_flux}}{=}& \sum_{r,m}Q^q_{rm} \big(\psi_t(U_n^r)-\psi_t(U_n^m)\big) + V_n^qU_n^q - V_n^0U_{n-1}^q \\
 \overset{\cref{eq:interp_lag_unite_deriv}}{=}& - \psi_t(U_n^q) + \psi_t(U_n^0) + V_n^qU_n^q - V_n^0U_{n-1}^q \\
 \overset{\cref{eq:EC_time_flux-b}}{=}& \vartheta_n^q - \vartheta_{n-1}^q + \vartheta_{n-1}^q - \vartheta_n^0 - V_n^0\big(U_{n-1}^q-U_n^0\big) \geq \vartheta_n^q - \vartheta_{n-1}^q
\end{align*}

\noindent by convexity of the entropy, $\vartheta^+-\vartheta^-\geq v^-(u^+-u^-)$, which completes the proof.
\end{proof}

\begin{theorem}\label{th:exist_uniq_sol_stDGSEM_GV}
The space-time DGSEM \cref{eq:space_time_DGSEM} with the EC fluxes \cref{eq:entropy_conserv_flux,eq:EC_time_flux} defined for a given entropy pair $\big(\vartheta(u),{\bf g}(u)\big)$ with $\vartheta''(\cdot)>0$, and the graph viscosities in \cref{eq:graph_viscosity,eq:stDGSEM_time_discr} such that
\begin{equation}\label{eq:condition_on_dkappa_dn}
 d_\kappa \geq 4L_fL_u\max_{k\neq l}\frac{|D^p_{kl}|}{\omega_l} \quad \forall \kappa\in\Omega_h, \quad d_n\geq2L_u\max_{r\neq m}\frac{|D^q_{rm}|}{\omega_m} \quad \forall n\geq0,
\end{equation}

\noindent with $L_f$ defined in \cref{eq:flux_lipschitz_cst} and $L_u$ defined in \cref{eq:condition_on_dkappa}, has a unique solution satisfying the maximum principle: $m\leq U_\kappa^{ij}(t_n^r)\leq M$ for all $\kappa\in\Omega_h$, $0\leq i,j\leq p$, and $0\leq r\leq q$.
\end{theorem}

\begin{proof}
\noindent\textit{Preliminaries.} We here extend to \cref{eq:space_time_DGSEM} techniques used in the proof of \cref{th:exist_uniq_sol_BE_GV} for the scheme \cref{eq:fully-discr_DGSEM}. We will again use notations like $U_n^r=U_\kappa^{ij}(t_n^r)$ when there is no ambiguity. Let us add the following regularization to $T_\kappa^r(u_h)$ in \cref{eq:stDGSEM_time_discr}
\begin{equation*}
 \nu\omega_r\omega_0\big(U_\kappa^{ij}(t_n^r)-U_\kappa^{ij}(t_{n-1}^q)\big),
\end{equation*}
\noindent with $\nu>0$ and we will then pass to the limit $\nu\to0$ (see last part of the proof). This term extends the stencil of the graph viscosity in $T_\kappa^r(u_h)$ and will be used only to prove uniqueness, while existence and stability will be shown to hold whatever $\nu\geq0$.

\noindent{\it Existence.} Using \cref{eq:EC_time_flux-a}, we rewrite the time discretization terms \cref{eq:stDGSEM_time_discr} in \cref{eq:space_time_DGSEM} as
\begin{align}
 T_\kappa^r =& \sum_m\omega_r\Big(\big(d_n\omega_m+2D^q_{rm}(1-2\beta_{rm})\big)U_n^r - (d_n\omega_m-4D^q_{rm}\beta_{rm})U_n^m\Big)  \nonumber\\
 & + \delta_{r0}\big(U_n^0-U_{n-1}^q\big) + \nu\omega_r\omega_0\big(U_n^r-U_{n-1}^q\big) \nonumber\\
 \overset{\cref{eq:interp_lag_unite_deriv}}{=}& \sum_m\gamma_n^{rm}\big(U_n^r-U_n^m\big)+(\delta_{r0}+\nu\omega_r\omega_0)\big(U_n^r-U_{n-1}^q\big), \label{eq:T_kappa_r_as_cvx_comb}
\end{align}

\noindent with $\gamma_n^{rm}=\omega_r(d_n\omega_m-4D^q_{rm}\beta_{rm})\geq0$ from \cref{eq:condition_on_dkappa_dn}, $\beta_{rm}=\beta(U_n^r,U_n^m)\in(0,\tfrac{1}{2}]$ defined in \cref{eq:def_beta_in_Uec}, and $L_u\geq1$ in \cref{eq:condition_on_dkappa}. By $\sum_m$ we denote $\sum_{m=0}^q$. 

We now rewrite the space discretization terms, $R_\kappa^{ij}+V_\kappa^{ij}$, in \cref{eq:space_time_DGSEM} as in \cref{eq:def_Unp_mp_as_convex_comb} and using \cref{eq:T_kappa_r_as_cvx_comb} we again introduce the following iterations for \cref{eq:space_time_DGSEM} divided by $\omega_i\omega_jJ_\kappa^{ij}$:
\begin{align}
\bigg(\delta_{r0}\!&+\!\nu\omega_r\omega_0 \!+\!\! \sum_{m=0}^q \gamma_n^{rm} \!+\! \frac{\omega_r}{2}\Big(\sum_{k=0}^p(\gamma_\kappa^{(ik)j}\!+\!\gamma_\kappa^{i(jk)})\!+\!\!\sum_{e\in\partial\kappa}\sum_{k=0}^p\gamma_e^k\Big)\bigg)^{(\mrm)}\!\!U_\kappa^{ij}(t_n^r)^{(\mrm+1)} \!=\!\!  \nonumber\\ 
 & (\delta_{r0}+\nu\omega_r\omega_0)U_\kappa^{ij}(t_{n-1}^q) + \sum_{m=0}^q \gamma_n^{rm}U_\kappa^{ij}(t_n^m)^{(\mrm)} \nonumber\\
 &+ \frac{\omega_r}{2}\sum_{k=0}^p\Big(\gamma_\kappa^{(ik)j}{\cal U}(U_\kappa^{ij},U_\kappa^{kj},\tfrac{{\bf n}_\kappa^{(ik)j}}{|{\bf n}_\kappa^{(ik)j}|})+\gamma_\kappa^{i(jk)}{\cal U}(U_\kappa^{ij},U_\kappa^{ik},\tfrac{{\bf n}_\kappa^{i(jk)}}{|{\bf n}_\kappa^{i(jk)}|})\Big)^{(\mrm)} \nonumber\\ &+ \frac{\omega_r}{2}\sum_{e,k}\gamma_e^k \Big(U_\kappa^{ij}-\frac{h\big(U_\kappa^{ij},u_h^{+}({\bf x}_e^{k},t),{\bf n}_e^k\big)-{\bf f}_\kappa^{ij}\cdot{\bf n}_e^k}{2L_f}\Big)^{(\mrm)} \label{eq:def_Unp_mp_stDGSEM_as_convex_comb}
\end{align}

\noindent with $\mrm\geq0$, the initial data $U_\kappa^{ij}(t_{n}^{0\leq r\leq q})^{(\mrm=0)}=U_\kappa^{ij}(t_{n-1}^q)$, and the coefficients defined above and in \cref{eq:conv_coeff_in_Un+1}. The iteration is a convex combination of quantities in $[m,M]$ whatever $\nu\geq0$ and one may again use the Schauder fixed point theorem to show that the map $u_h(t_n^{0\leq r\leq q})^{(\mrm)}\mapsto u_h(t_n^{0\leq r\leq q})^{(\mrm+1)}$ defined by \cref{eq:def_Unp_mp_stDGSEM_as_convex_comb} has at least one fixed point with values in $[m,M]$ which is solution to \cref{eq:space_time_DGSEM}.


\noindent{\it Uniqueness.} We again assume two solutions $u_h$ and $w_h$ of \cref{eq:space_time_DGSEM} with the same initial data at time $t_{n-1}^q$ and show that necessarily $u_h\equiv w_h$. Subtracting both schemes gives
\begin{equation*}
 \omega_i\omega_jJ_\kappa^{ij}\big(T_\kappa^r(u_h)-T_\kappa^r(w_h)\big)_{{\bf x}_\kappa^{ij}} + \frac{\omega_r \Delta t^{(n)}}{2}\big(R_\kappa^{ij}(u_h)+V_\kappa^{ij}(u_h)-R_\kappa^{ij}(w_h)-V_\kappa^{ij}(w_h)\big)_{t_n^r}=0.
\end{equation*}

Using SBP \cref{eq:SBP}, we rewrite the time discretization terms as 
\begin{align*}
 T_\kappa^r(u_h) =& \delta_{rq}U_n^q - \delta_{r0}U_{n-1}^q + \nu\omega_r\omega_0(U_n^r-U_{n-1}^q) \\ &+ \sum_m(Q^q_{rm}-Q^q_{mr})U_{ec}\big(U_n^r,U_n^m\big) + d_n\omega_r\omega_m\big(U_n^r-U_n^m\big),
\end{align*}

\noindent with again $\sum_m\equiv\sum_{m=0}^q$, and obtain
\begin{align*}
 T_\kappa^r(u_h)-T_\kappa^r(w_h) &= (\delta_{rq}+\nu\omega_r\omega_0)(U_n^r -W_n^r) \\ 
&+ \sum_mQ^q_{rm}\big(U_{ec}(U_n^r,U_n^m) - U_{ec}(W_n^r,U_n^m)\big)+\tfrac{1}{2}d_n\omega_r\omega_m(U_n^r-W_n^r) \\
&+ \sum_mQ^q_{rm}\big(U_{ec}(W_n^r,U_n^m) - U_{ec}(W_n^r,W_n^m)\big)-\tfrac{1}{2}d_n\omega_r\omega_m(U_n^m-W_n^m) \\
&- \sum_mQ^q_{mr}\big(U_{ec}(U_n^r,U_n^m) - U_{ec}(U_n^r,W_n^m)\big)+\tfrac{1}{2}d_n\omega_r\omega_m(U_n^m-W_n^m) \\
&- \sum_mQ^q_{mr}\big(U_{ec}(U_n^r,W_n^m) - U_{ec}(W_n^r,W_n^m)\big)-\tfrac{1}{2}d_n\omega_r\omega_m(U_n^r-W_n^r).
\end{align*}

Since $\partial_1U_{ec}(u^-,u^+)\in[0,L_u]$, the two first and last rows in the RHS have the same sign under \cref{eq:condition_on_dkappa_dn}, we can thus combine this result together with \cref{eq:ineq_with_a-b_coeffs} to get 
\begin{equation*}
 \sum_{\kappa,i,j,r}a_\kappa^{ijr}|U_\kappa^{ij}(t_n^r)-W_\kappa^{ij}(t_n^r)| \leq \sum_{\kappa,i,j,r} b_\kappa^{ijr}|U_\kappa^{ij}(t_n^r)-W_\kappa^{ij}(t_n^r)|
\end{equation*}

\noindent with $a_\kappa^{ijr}=(\delta_{rq}+\nu\omega_r\omega_0)a_\kappa^{ij}(t_n^r)$ and
\begin{equation*}
 b_\kappa^{ijr} = \frac{\omega_r}{2}b_\kappa^{ij}(t_n^r)+\frac{\omega_i\omega_jJ_\kappa^{ij}}{\Delta t^{(n)}}\sum_{m}\big(d_n\omega_r\omega_m+2|Q_{rm}^q|L_{u}\big)|\varphi_\kappa^{ijm}-\varphi_\kappa^{ijr}|,
\end{equation*}
\noindent where the $a_\kappa^{ij}$ and $b_\kappa^{ij}$ are defined in \cref{eq:ineq_with_a-b_coeffs}. As in \cref{eq:ineq_with_a-b_coeffs}, we define $\varphi_\kappa^{ijr}=\varphi_\kappa^{ij}(t_n^r)\geq0$ the average of $\varphi({\bf x},t)=e^{-\gamma(|{\bf x}|+t-t_n^0)}$ over the subcell $\kappa_{ij}\subset\kappa$, with $\gamma>0$, evaluated at $t_n^r$. Using \cref{eq:bounds_on_a-b_coeffs,eq:regular_mesh_cond} together with $\omega_i\omega_jJ_\kappa^{ij}\leq|\kappa|\leq h^d$ we obtain 
\begin{equation*}
 a_\kappa^{ijr} \!\geq\! \frac{\nu\omega_0^{d+2}\beta h^d}{2d\Delta t^{(n)}}\!\!\!\!\inf_{ B_\kappa(h)\times[t_n^0,t_n^q]}\!\!\!\!\varphi({\bf x},t), \; b_\kappa^{ijr}\!\leq\! h^d\Big(\frac{8d_\kappa+L_1+L_2}{d\beta}+\frac{8d_n}{\Delta t^{(n)}}\Big)\!\!\!\!\sup_{ B_\kappa(2h)\times[t_n^0,t_n^q]}\!\!\!\!\!\!|\nabla_{{\bf x},t}\varphi|,
\end{equation*}
\noindent and using $|\nabla_{{\bf x},t}\varphi|=|(\nabla\varphi^\top,\partial_t\varphi)^\top|=\sqrt{2}\gamma\varphi$, the choice $\gamma e^{\gamma(3h+2\Delta t^{(n)})}\big(\tfrac{8d_\kappa+L_1+L_2}{d\beta}+\tfrac{8d_n}{\Delta t^{(n)}}\big)<\tfrac{\nu\omega_0^{d+2}\beta}{2\sqrt{2}\Delta t^{(n)}}$ imposes $a_\kappa^{ijr}>b_\kappa^{ijr}$, so $U_\kappa^{ij}(t_n^r)=W_\kappa^{ij}(t_n^r)$ for all the DOFs. 

\textit{Uniqueness at $\nu=0$.} Existence and stability hold whatever $\nu\geq0$, but uniqueness requires $\nu>0$. We now justify that uniqueness can be extended to $\nu=0$. The scheme is a Lipschitz continuous function of $u_h$ and $\nu$. Applying the implicit function theorem for non-differentiable functions \cite{Jittorntrum_IFT_78}, for every $\nu>0$ there exists a unique $u_h = u_{h,\nu}$ and neighborhoods of $u_{h,\nu}$ and $\nu$ where $\nu\mapsto u_{h,\nu}$ is one-to-one and continuous. It is even a Cauchy-continuous function over $(0,\infty)$, since for all Cauchy sequence $(\nu_i)_i$ in $(0,\infty)$, $(u_{h,\nu_i})_i$ is a Cauchy sequence in $\ell_\infty(\mathbb{N})$. The latter being  a complete metric space, $u_{h,\nu}$ can be extended to $\nu=0$ by Cauchy-continuity. Since $u_h$ exists and is bounded at $\nu=0$, we obtain uniqueness in this limit. This concludes the proof.
\end{proof}

Once again a consequence is the following invariance property obtained from \cref{eq:def_Unp_mp_stDGSEM_as_convex_comb} with initial data  $U_\kappa^{ij}(t_{n}^{0\leq r\leq q})^{(\mrm=0)}=U_\kappa^{ij}(t_{n-1}^q)$ and uniqueness of the solution to \cref{eq:space_time_DGSEM}.

\begin{corollary}\label{th:corollary_MP_st_DGSEM}
 Under the assumptions of \cref{th:exist_uniq_sol_stDGSEM_GV}, the solution to space-time DGSEM \cref{eq:space_time_DGSEM} with the graph viscosities in \cref{eq:graph_viscosity,eq:stDGSEM_time_discr} satisfies the estimate
\begin{equation*}
 \min_{\kappa'\in\Omega_h}\min_{0\leq k,l\leq p} U_{\kappa'}^{kl,n} \leq U_\kappa^{ij}(t_n^r) \leq \max_{\kappa'\in\Omega_h}\max_{0\leq k,l\leq p} U_{\kappa'}^{kl,n} \quad\forall \kappa\in\Omega_h, 0\leq i,j\leq p, 0\leq r\leq q.
\end{equation*}
\end{corollary}

\begin{theorem}\label{th:ES_sol_st_DGSEM}
Under the assumptions of \cref{th:exist_uniq_sol_stDGSEM_GV}, the space-time DGSEM \cref{eq:space_time_DGSEM} with the EC fluxes \cref{eq:entropy_conserv_flux,eq:EC_time_flux}, defined for a given entropy pair $\big(\vartheta(u),{\bf g}(u)\big)$ with $\vartheta''(\cdot)>0$, and the graph viscosities in \cref{eq:graph_viscosity,eq:stDGSEM_time_discr} satisfies the following inequality for every entropy pair $\big(\eta(u),{\bf q}(u)\big)$ with $\eta''(\cdot)\geq0$:
\begin{align*}
 \langle\eta\big(u_h(t_{n}^q)\big)\rangle_\kappa \leq& \langle\eta\big(u_h(t_{n-1}^q)\big)\rangle_\kappa \\&- \frac{\Delta t^{(n)}}{|\kappa|}\sum_{r=0}^q\frac{\omega_r}{2}\sum_{e\in\partial\kappa}\sum_{k=0}^p\omega_kJ_e^k Q\big(u_h^-({\bf x}_e^{k},t_n^r),u_h^+({\bf x}_e^{k},t_n^r),{\bf n}_e^k\big),
\end{align*}
\noindent with the numerical entropy flux $Q(\cdot,\cdot,\cdot)$ defined in \cref{eq:ES_interface_flux}.
\end{theorem}

\begin{proof}
 The proof follows the lines of the proof of \cref{th:ES_sol_BE_GV} by applying the convex entropy $\eta(\cdot)$ to the limit $\mrm\to\infty$ of the convex combination \cref{eq:def_Unp_mp_stDGSEM_as_convex_comb} with $\nu=0$ and further using \cref{eq:ES_interface_flux,eq:hec_over_fan-b}, so we omit the details. Compared to \cref{th:ES_sol_BE_GV}, the only additional terms to the RHS come from the time contributions. Summing them over $0\leq r\leq q$ and using the definition of $\gamma_n^{rm}=\omega_r(d_n\omega_m-4D^q_{rm}\beta_{rm})$ and using the notation $\sum_{r,m}\equiv\sum_{r,m=0}^q$, they read
\begin{align*}
 -\eta_n^0&+\eta_{n-1}^q-\sum_{r,m}\omega_r(d_n\omega_m-4D^q_{rm}\beta_{rm})(\eta_n^r-\eta_n^m) = -\eta_n^0+\eta_{n-1}^q  \\ +& 4\sum_{r,m}(Q^q_{rm}\beta_{rm}-Q^q_{mr}\beta_{mr})\eta_n^r \overset{\cref{eq:interp_lag_unite_deriv}}{\underset{\cref{eq:def_beta_in_Uec}}{=}} -\eta_n^0+\eta_{n-1}^q  - 4\sum_{r,m}(Q^q_{rm}+Q^q_{mr})\beta_{mr}\eta_n^r \\ &\overset{\cref{eq:SBP}}{=}  -\eta_n^q+\eta_{n-1}^q,
\end{align*}

\noindent where we have switched indices $r\leftrightarrow m$ and then used $\beta_{rm}=\tfrac{1}{2}-\beta_{mr}$ and $\beta_{00}=\beta_{qq}=\tfrac{1}{4}$ from the definition \cref{eq:def_beta_in_Uec} of $\beta_{rm}=\beta(U_n^r,U_n^m)$. This completes the proof.
\end{proof}

\begin{remark}[Kru\v{z}kov's entropies] 
The DGSEM in \cref{th:ES_BE_DGSEM_noGV,th:ES_stDGSEM_noGV} without graph viscosity is not ES for the Kru\v{z}kov's entropies \cite{Krukov1970FIRSTOQ}, $\vartheta(u)=|u-K|$, ${\bf g}(u)=\text{sgn}(u-K)\big({\bf f}(u)-{\bf f}(K)\big)$, $K\in\mathbb{R}$, since it requires the mapping $v(u)=\vartheta'(u)$ to be one-to-one, but $v'(u)=\delta(u-K)$. Applying \cref{eq:entropy_conserv_flux,eq:EC_time_flux-b}, we obtain $h_{ec}(u^-,u^+,{\bf n})={\bf f}(K)\cdot{\bf n}$ and $U_{ec}(u^-,u^+)={K}$ which are not consistent. Using the graph viscosity and EC fluxes defined for one strictly convex function, e.g., $\vartheta(u)=\tfrac{u^2}{2}$, however applies to them as it only requires $\eta''(\cdot)\geq0$ and, in particular, \cref{th:ES_sol_BE_GV,th:ES_sol_st_DGSEM} hold for $\eta(u)=|u-K|$.
\end{remark}

\begin{remark}[bounds in limiters] 
Adding graph viscosity to the schemes \cref{eq:fully-discr_DGSEM,eq:space_time_DGSEM} annihilates they high-order accuracy. These low-order schemes may however constitute building-blocks of limiters such as the flux-corrected transport limiter \cite{BORIS_Book_FCT_73,zalesak1979fully,MRR_BEDGSEM_23}, or the convex limiter \cite{Guermond_etal_IDP_conv_lim_18,PAZNER_idg_DGSEM20211}. The properties of the present schemes indeed allow to apply many ways to limit the solution to impose different bounds to the high-order solution such as the maximum principle in \cref{th:exist_uniq_sol_BE_GV,th:exist_uniq_sol_stDGSEM_GV}, the invariance properties in \cref{th:corollary_MP_BE_GV,th:corollary_MP_st_DGSEM}, or entropy inequalities based on any convex entropy in \cref{th:ES_sol_BE_GV,th:ES_sol_st_DGSEM}. This is left for future work and we here consider another strategy where we locally adapt the graph viscosity from the local regularity of the solution (see \cref{sec:num_xp}).
\end{remark}

\begin{remark}[sparse low-order discretization] 
The accuracy of the schemes \cref{eq:fully-discr_DGSEM,eq:space_time_DGSEM} with graph viscosities may be improved by substituting the sparse discrete derivative matrix $\hat{\bf Q}^p$ for ${\bf Q}^p$ (and $\hat{\bf Q}^q$ for ${\bf Q}^q$ for the space-time DGSEM) with entries $\hat{Q}^p_{k,\min(k+1,p)}=-\hat{Q}^p_{k,\max(k-1,0)}=\tfrac{1}{2}$ for $0\leq k\leq p$. These operators satisfy properties \cref{eq:SBP,eq:interp_lag_unite_deriv} and \cref{eq:discr_Met_Id} providing the metric terms are modified when considering curved elements \cite[\S~3.1]{PAZNER_idg_DGSEM20211}. As a consequence, 
\cref{th:exist_uniq_sol_BE_GV,th:ES_sol_BE_GV,th:exist_uniq_sol_stDGSEM_GV,th:ES_sol_st_DGSEM} hold. 
\end{remark}

%
%
\section{Numerical experiments}\label{sec:num_xp}

In this section we present numerical experiments on problems involving discontinuous solutions in one space dimension. The objective is here to illustrate the properties of the DGSEM schemes \cref{eq:fully-discr_DGSEM,eq:space_time_DGSEM}, but also their dissipative character. We consider problems with either a convex, or a non convex flux:
\begin{equation}
 \partial_tu+\partial_xf(u)=0 \text{ in } [0,1]\times(0,\infty), \quad u(\cdot,0)=u_0(\cdot) \text{ in } [0,1],
\end{equation}

\noindent with appropriate boundary conditions at $x=0$ and $x=1$. The different problems are defined in \cref{tab:def_pbs}. We define the EC space and time fluxes \cref{eq:entropy_conserv_flux,eq:EC_time_flux} with the square entropy $\vartheta(u)=\tfrac{u^2}{2}$, so $L_u=1$ in \cref{eq:condition_on_dkappa,eq:condition_on_dkappa_dn} (see \cref{rk:square_entropy1,rk:square_entropy2}), and use the Godunov solver \cite{godunov_59} at interfaces.
We use a formally fourth-order scheme, $p=q=3$ in \cref{eq:space_time_DGSEM}, except problem 2 that is steady-state and is solved with the backward Euler DGSEM \cref{eq:fully-discr_DGSEM} and $p=3$. The time step is computed from $\tfrac{\Delta t^{(n)}}{\text{diam}\,\kappa}L_f \leq \CFL$, where $\CFL=1$ for an unsteady problem, and $\CFL=10^3$ for a steady-state problem. The nonlinear algebraic systems \cref{eq:fully-discr_DGSEM,eq:space_time_DGSEM} are solved by using an iterative algorithm based on the exact linearization of the residuals at each subiteration. We will evaluate the effect of graph viscosity by comparing three different computations with schemes \cref{eq:space_time_DGSEM,eq:fully-discr_DGSEM}: 
\begin{description}
 \item[no graph viscosity:] the graph viscosity is removed by setting $d_\kappa=d_n=0$;
 \item[graph viscosity:] the graph viscosity coefficients are evaluated from their theoretical values \cref{eq:condition_on_dkappa,eq:condition_on_dkappa_dn}; 
 \item[adapted graph viscosity:] we evaluate the Persson and Peraire troubled cell indicator \cite{persson_peraire_lim_06} at each time step to tune the graph viscosity coefficients. The indicator is evaluated from $u_h^{(n)}$ expressed in the basis of Legendre polynomials by first computing ${\cal S}_\kappa=\log\big(\|u_h-\Pi_h^{p-1}u_h\|_{L^2(\kappa)}^2/\|u_h\|_{L^2(\kappa)}^2\big)$ with $\Pi_h^{p-1}$ the projector onto ${\cal V}_h^{p-1}$. Setting ${\cal S}_0=-4\log(2p)$, $d_\kappa$ and $d_n$ are multiplied by: 0 if ${\cal S}_\kappa<{\cal S}_0-\tfrac{1}{10}$; 1 if ${\cal S}_\kappa>{\cal S}_0+\tfrac{1}{10}$; $\tfrac{1}{2}+\tfrac{1}{2}\sin\big(5\pi({\cal S}_\kappa-{\cal S}_0)\big)$ else.
\end{description}

\begin{table}
     \centering
     \caption{\label{tab:RP_IC} Definitions of the different problems and numerical parameters.}
     \begin{tabular}{lllccccccc}
        \noalign{\smallskip}\hline\noalign{\smallskip}
        \# & $f(u)$ & $u_0(x)$ & $u(0,t)$ & $u(1,t)$ & $t$ & $p$ & $q$ & $\text{diam}\,\kappa$ & $\CFL$ \\
        \noalign{\smallskip}\hline\noalign{\smallskip}
        1 & $\frac{u^2}{2}$ & $\sin(2\pi x)$ & $u(1,t)$ & $-$ & $0.4$ & $3$ & $3$ & $\frac{1}{40}$ & $1$ \\
        2 & $\frac{u^2}{2}$ & $1-2x$ & $1$ & $-1$ & $-$ & $3$ & $0$ & $\frac{1}{40}$ & $10^3$ \\
        3 & $\frac{u^2}{2}$ & $1+\sin(2\pi x)$ & $u(1,t)$ & $-$ & $\frac{3}{4\pi}$ & $3$ & $3$ & $\frac{1}{40}$ & $1$ \\
        4 & $\frac{u^2}{u^2+\frac{(1-u)^2}{2}}$ & $1_{x<0.5}$ & $1$ & $0$ & $0.2$ & $3$ & $3$ & $\frac{1}{40}$ & $1$ \\
        5 & $\frac{u^2}{u^2+\frac{(1-u)^2}{4}}$ & $3(1_{x>0.5}-1_{x<0.5})$ & $-3$ & $3$ & $1$ & $3$ & $3$ & $\frac{1}{40}$ & $1$ \\
        \noalign{\smallskip}\hline\noalign{\smallskip}
    \end{tabular}
    \label{tab:def_pbs}
\end{table} 

Results are displayed in \cref{fig:burgers} for the inviscid Burgers' equation and in \cref{fig:bukley_leverett} for the Buckley-Leverett equation. The DGSEM schemes without graph viscosities are ES for the square entropy $\vartheta(u)=\tfrac{u^2}{2}$ (see \cref{th:ES_BE_DGSEM_noGV,th:ES_stDGSEM_noGV}), so we expect them to capture the entropy weak solution when the flux is convex \cite{Panov_1994} as highlighted in \cref{fig:burgers}. This is however no longer the case for the Buckley-Leverett equation in \cref{fig:bukley_leverett}, where we observe nonphysical weak solutions. Adding graph viscosity allows to capture the entropy weak solution, but results in large dissipation smearing the discontinuities, in particular for problems 3 and 4. This dissipative character may be illustrated from a comparison with the scaling limiter \cite{zhang2012maximum}. Consider, e.g., scheme  \cref{eq:fully-discr_DGSEM} that we rewrite in one space dimension as
\begin{equation*}
 \Big(1+\frac{\Delta t^{(n)}}{J_\kappa^i}d_\kappa\Big)U_\kappa^{i,n+1}=\frac{\Delta t^{(n)}}{J_\kappa^i}d_\kappa\langle u_h^{(n+1)}\rangle_\kappa + \tilde{U}_\kappa^{i,n+1} \quad \forall 0\leq i\leq p, 
\end{equation*}

\noindent with $J_\kappa^i=\tfrac{1}{2}\text{diam}\,\kappa$, $\langle u_h^{(n+1)}\rangle_\kappa=\tfrac{1}{2}\sum_{l=0}^p\omega_lU_\kappa^{l,n+1}$ the cell-averaged solution, and $\tilde{U}_\kappa^{i,n+1}=U_\kappa^{i,n}-\tfrac{\Delta t^{(n)}}{J_\kappa^i\omega_i}R_\kappa^i(u_h^{(n+1)})$ which may be viewed as an unlimited solution. We therefore have $U_\kappa^{i,n+1}=(1-\theta_\kappa)\langle u_h^{(n+1)}\rangle_\kappa+\theta_\kappa\tilde{U}_\kappa^{i,n+1}$ with $\theta_\kappa=\big(1+\tfrac{\Delta t^{(n)}}{J_\kappa^i}d_\kappa\big)^{-1}\in(0,1)$ which corresponds to the linear scaling limiter of $u_h^{(n+1)}$ around its cell average \cite{zhang2012maximum}. The limit for large viscosity, $\lim_{d_\kappa\to\infty}\theta_\kappa=0$, indeed corresponds to imposing $U_\kappa^{i,n+1}=\langle u_h^{(n+1)}\rangle_\kappa$ as may be observed in \cref{fig:burgers,fig:bukley_leverett}.

Finally, adapting locally the viscosity with the troubled cell indicator results in less dissipation and sharp resolution of the solution features, while successfully limiting the solution. We stress that in this case, the MPP property is no longer guaranteed as this requires graph viscosity estimates to hold in all mesh elements, nor are valid the existence and uniqueness of the solution. Local viscosity adaptation in implicit scheme is here empirical, as is often the case in high-order schemes.

\begin{figure}
\centering
\subfloat{\includegraphics[width=4.3cm]{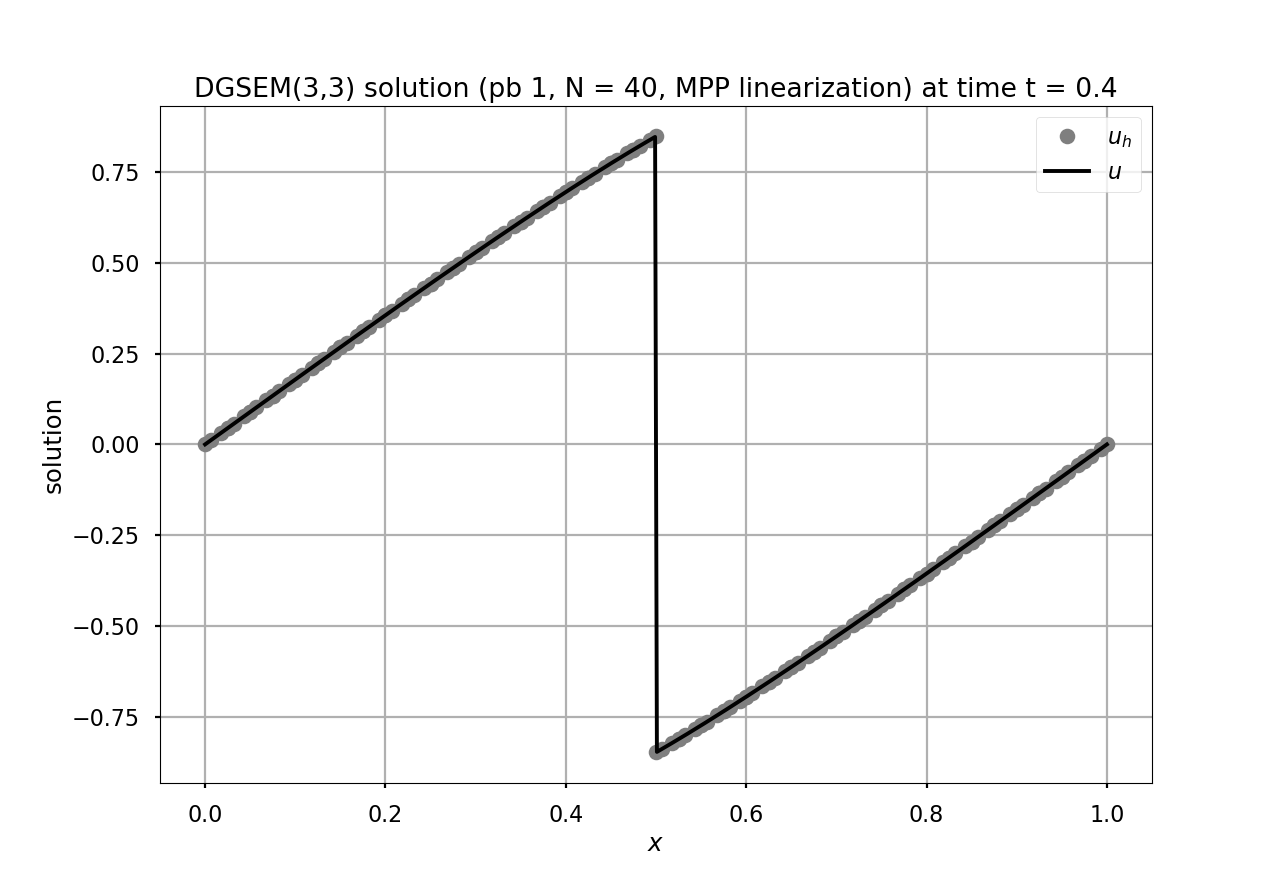}}%
\subfloat{\includegraphics[width=4.3cm]{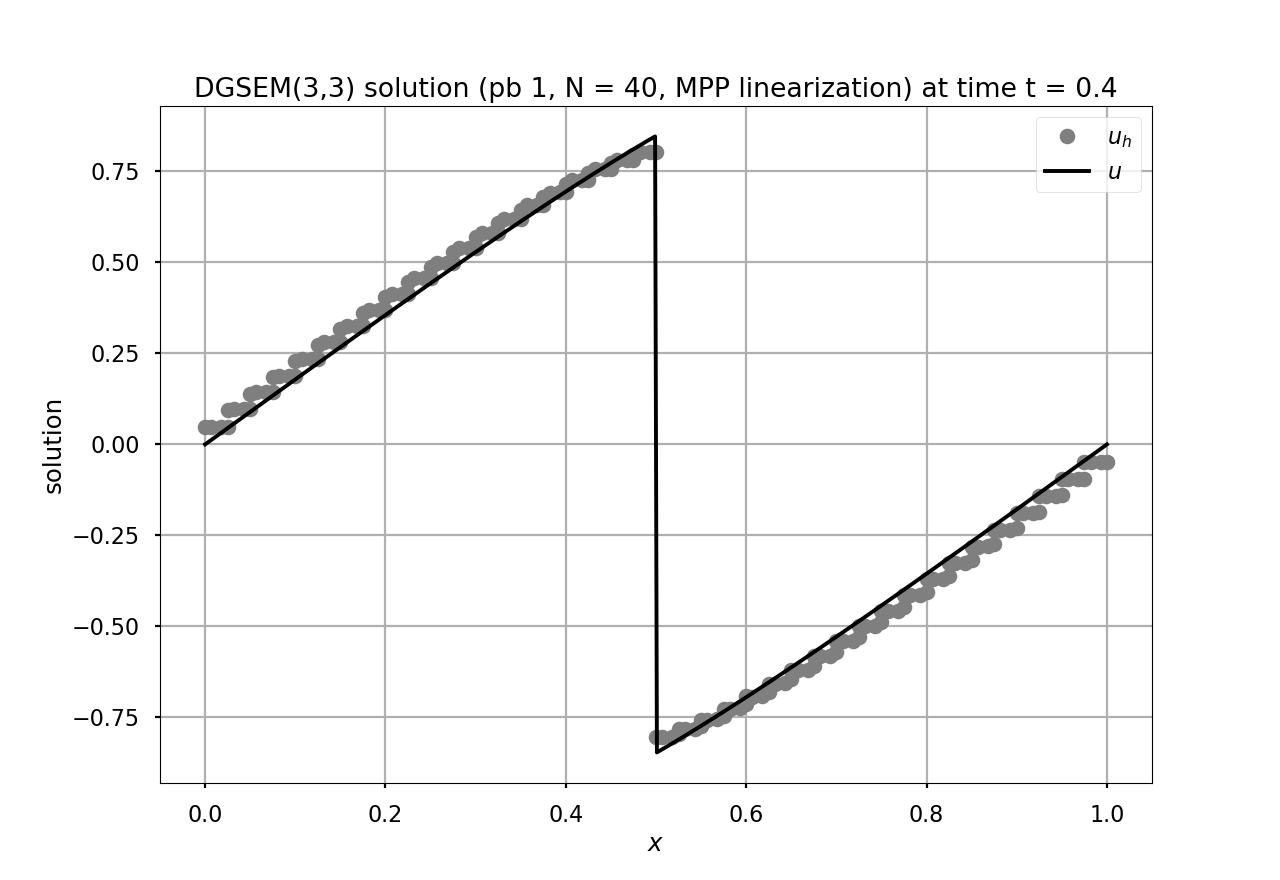}}%
\subfloat{\includegraphics[width=4.3cm]{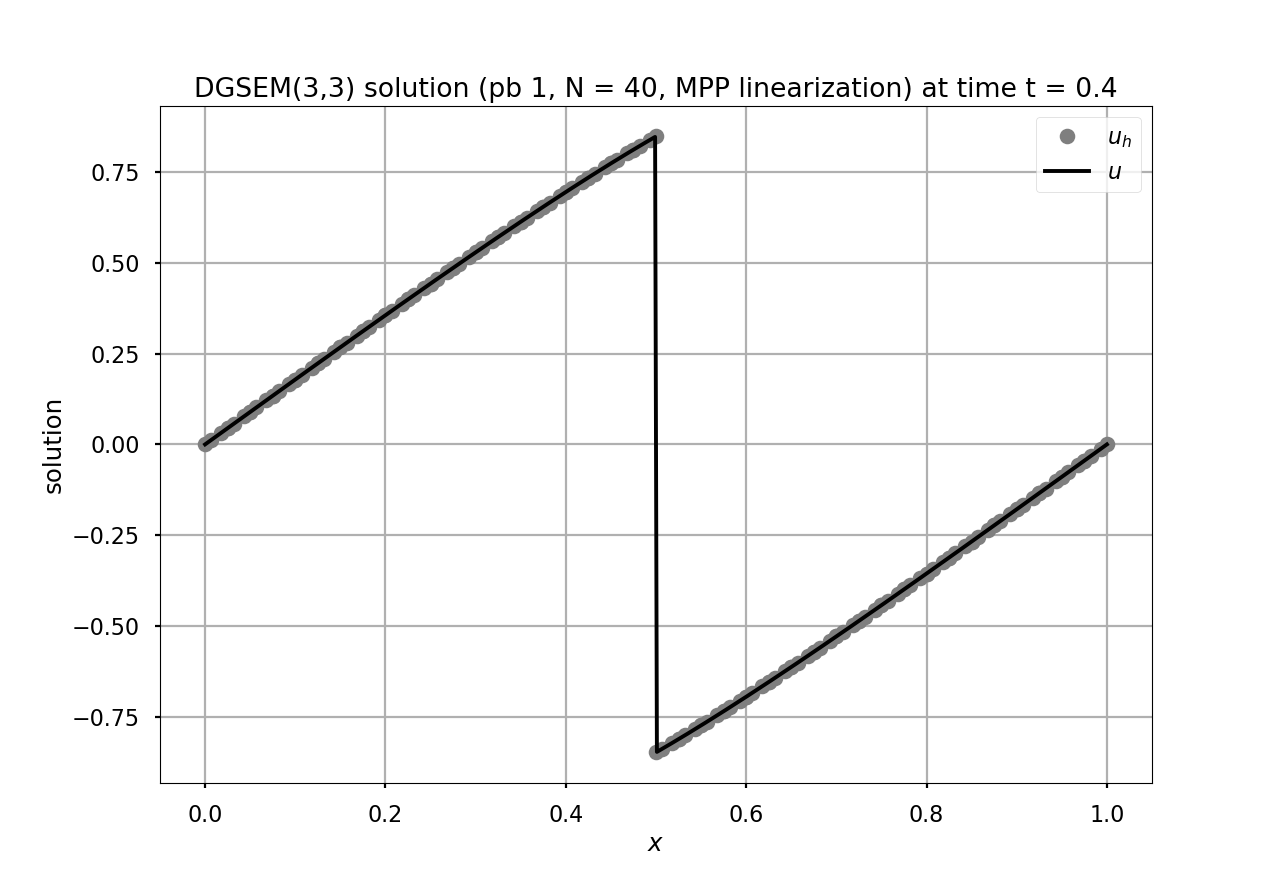}}\\\vspace{-0.25cm} \subfloat{\includegraphics[width=4.3cm]{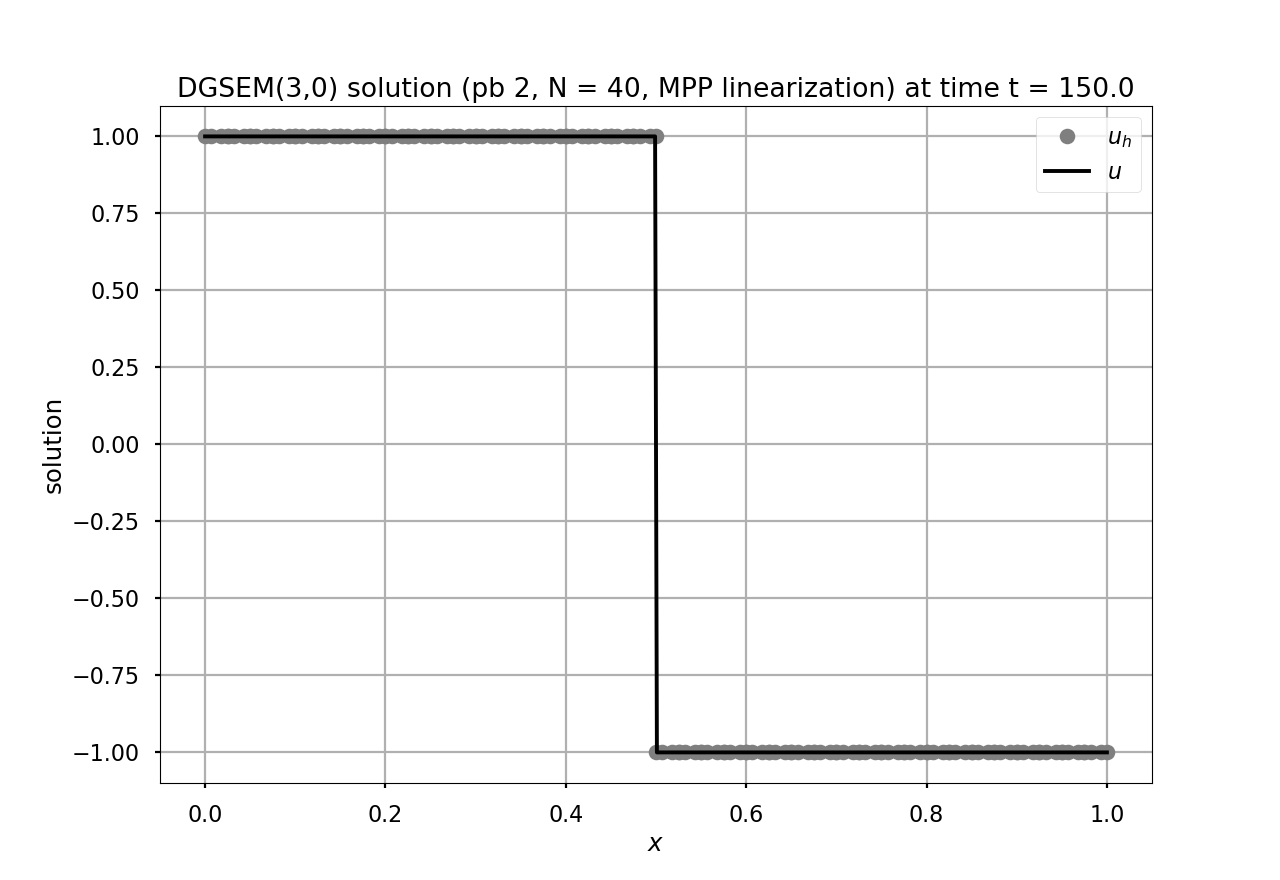}}%
\subfloat{\includegraphics[width=4.3cm]{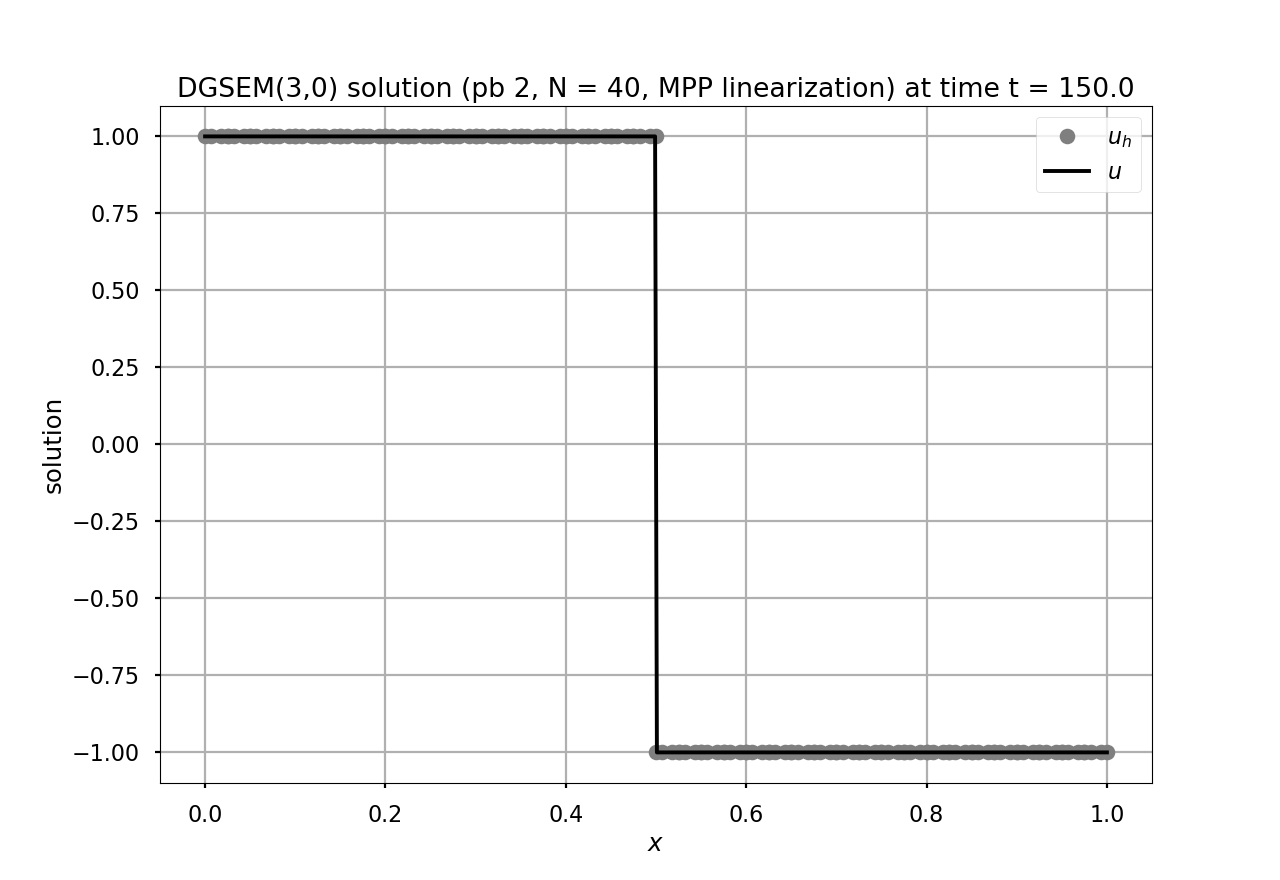}}%
\subfloat{\includegraphics[width=4.3cm]{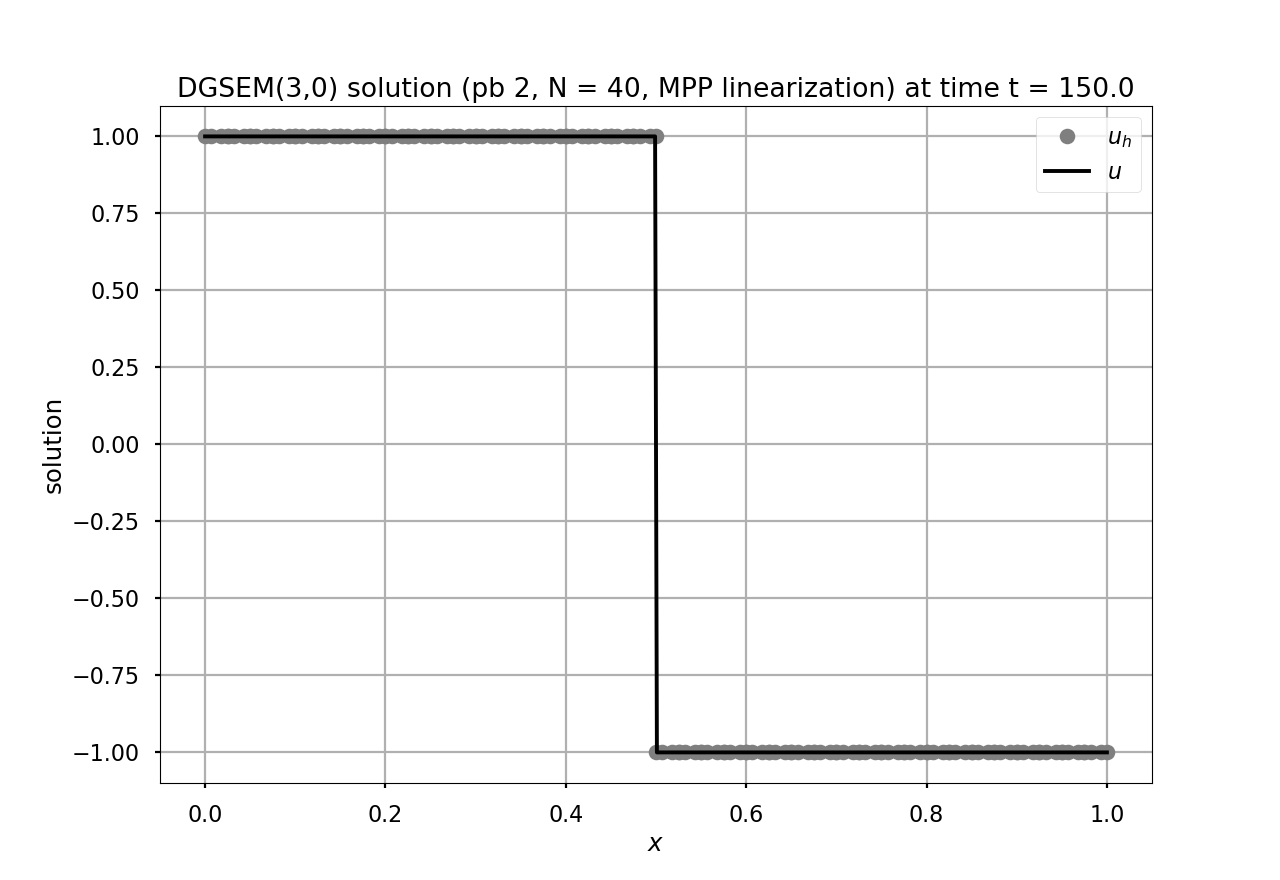}}\\\vspace{-0.25cm} 
\setcounter{subfigure}{0}
\subfloat[no graph viscosity]{\includegraphics[width=4.3cm]{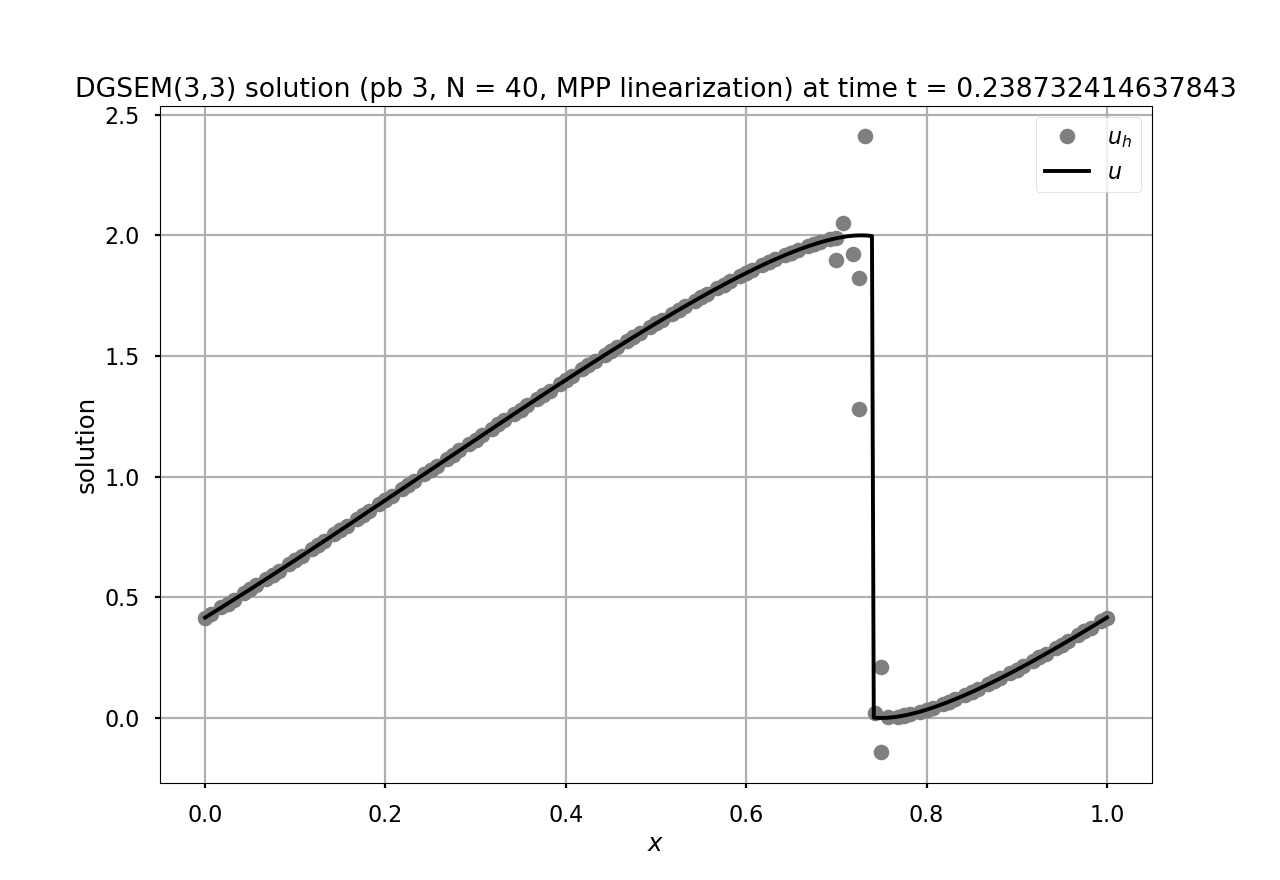}}
\subfloat[graph viscosity]{\includegraphics[width=4.3cm]{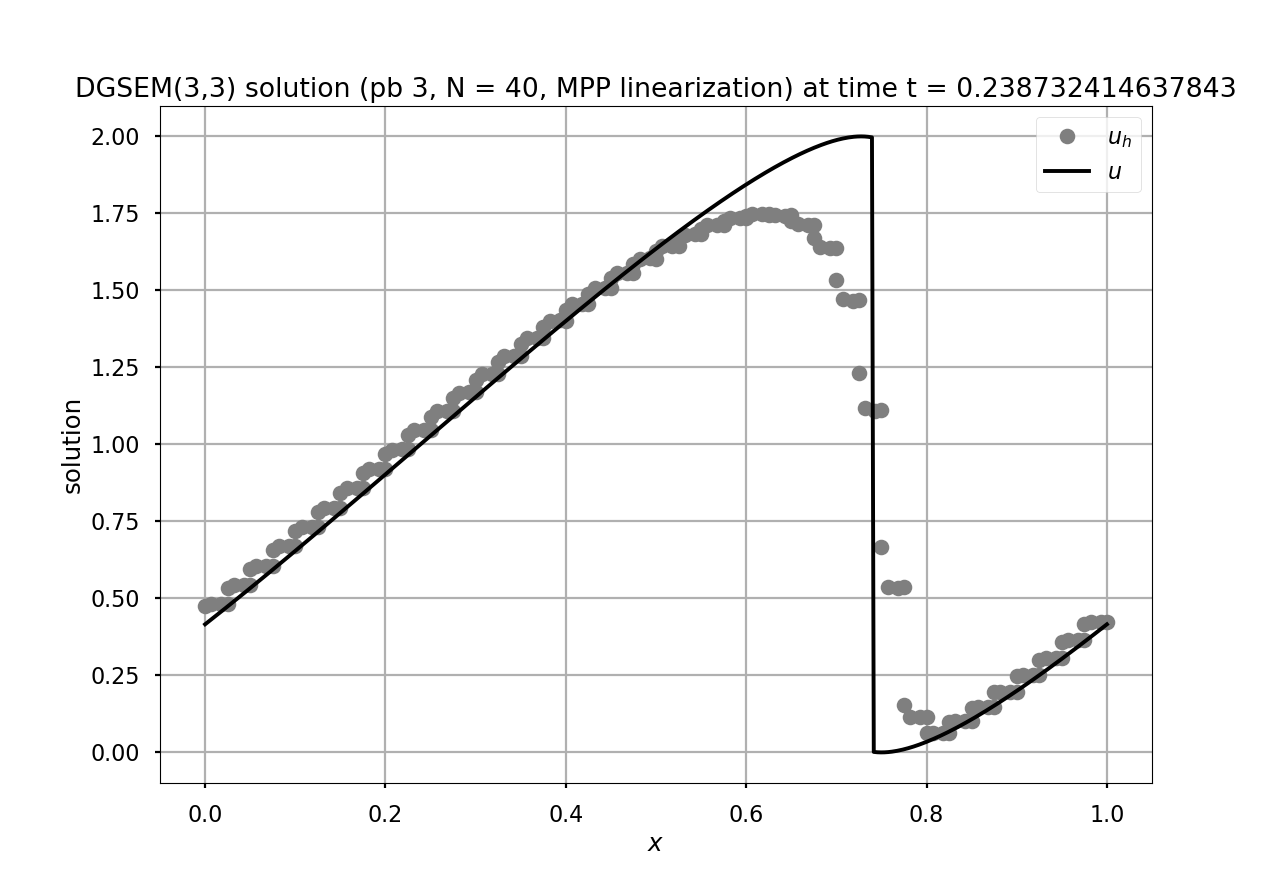}}
\subfloat[adapted graph viscosity]{\includegraphics[width=4.3cm]{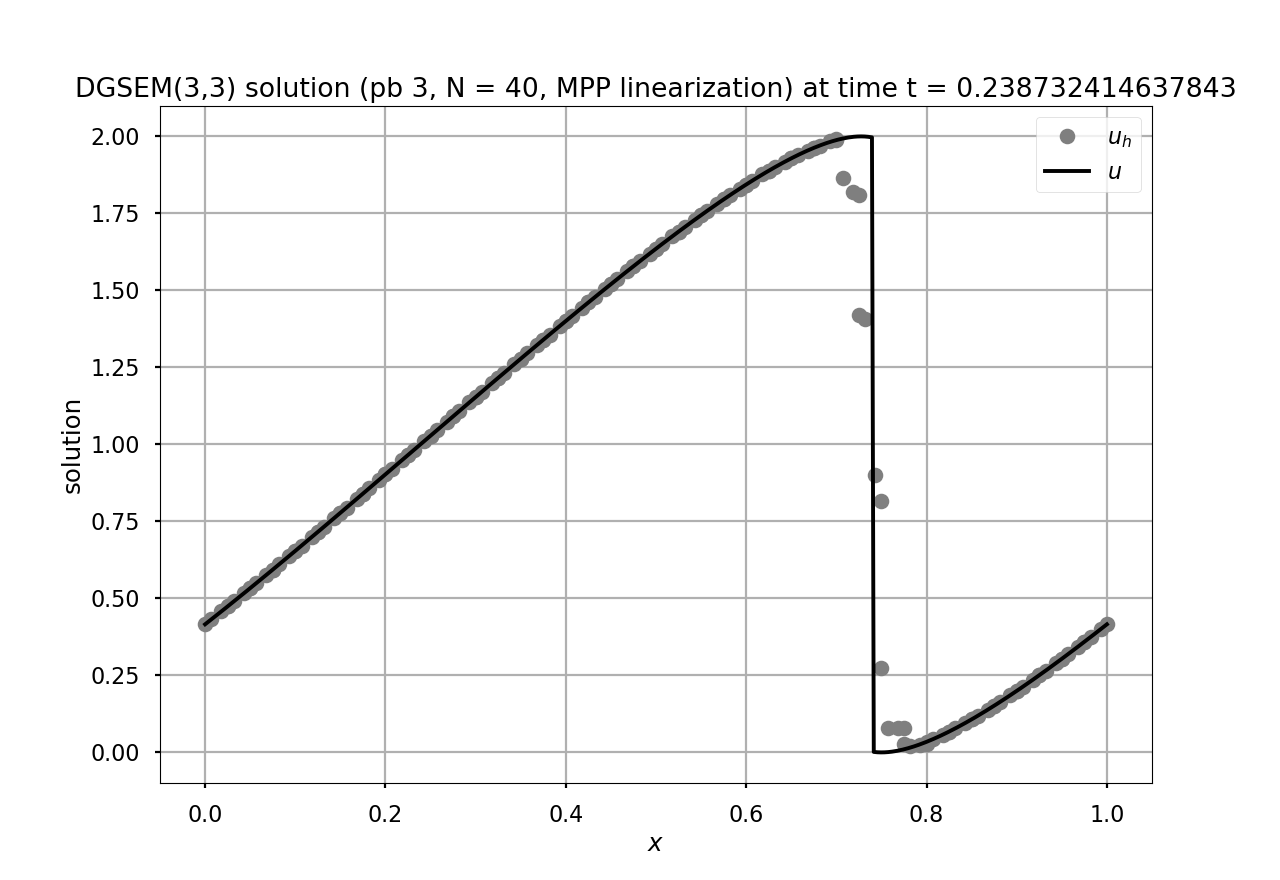}}
\caption{Inviscid Burgers' equation: solutions to problems 1 (top), 2 (middle), and 3 (bottom) defined in \cref{tab:def_pbs}. The $p+1=4$ DOFs per mesh element are displayed at the final time (bullets) and compared to the exact solution (lines).}
\label{fig:burgers}
\end{figure} 

\begin{figure}
\centering
\subfloat{\includegraphics[width=4.3cm]{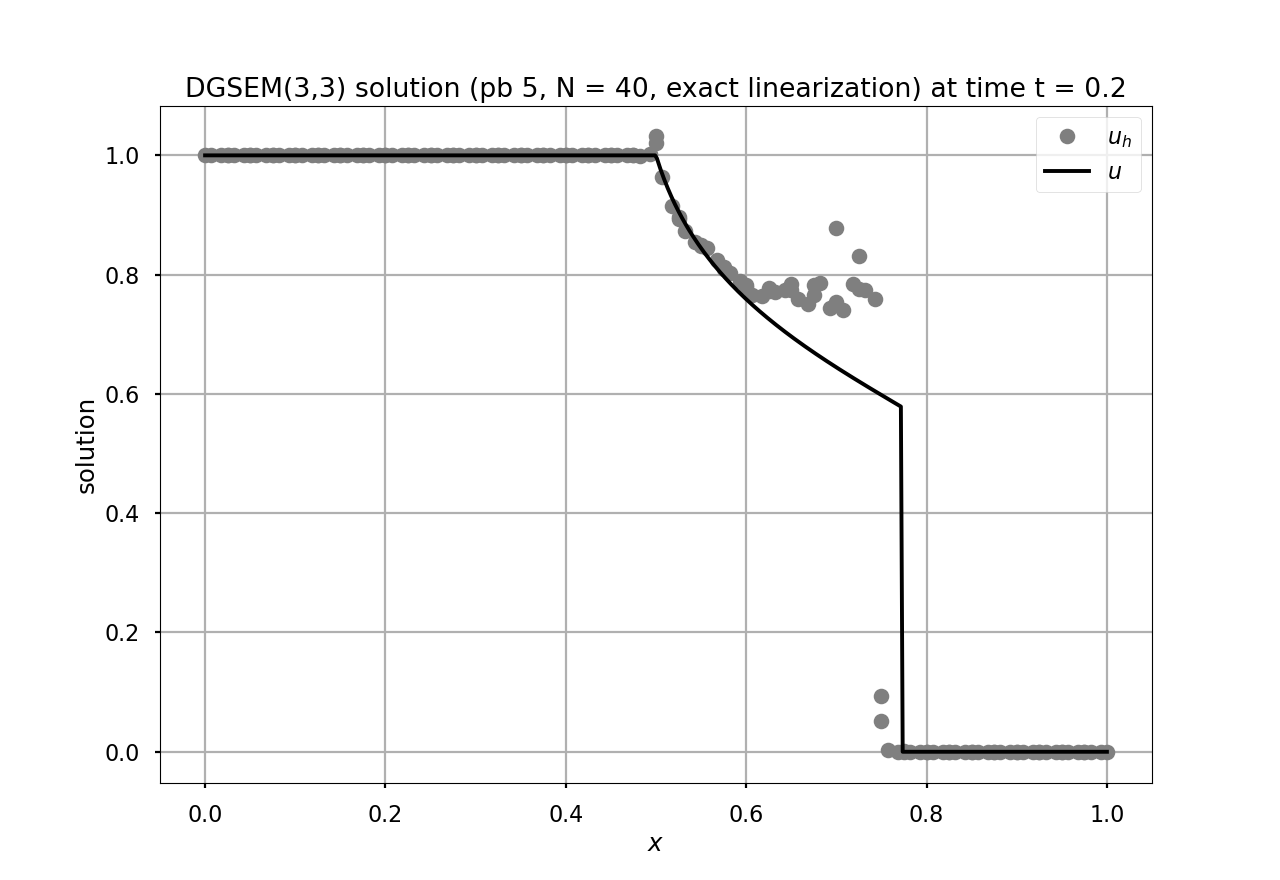}}%
\subfloat{\includegraphics[width=4.3cm]{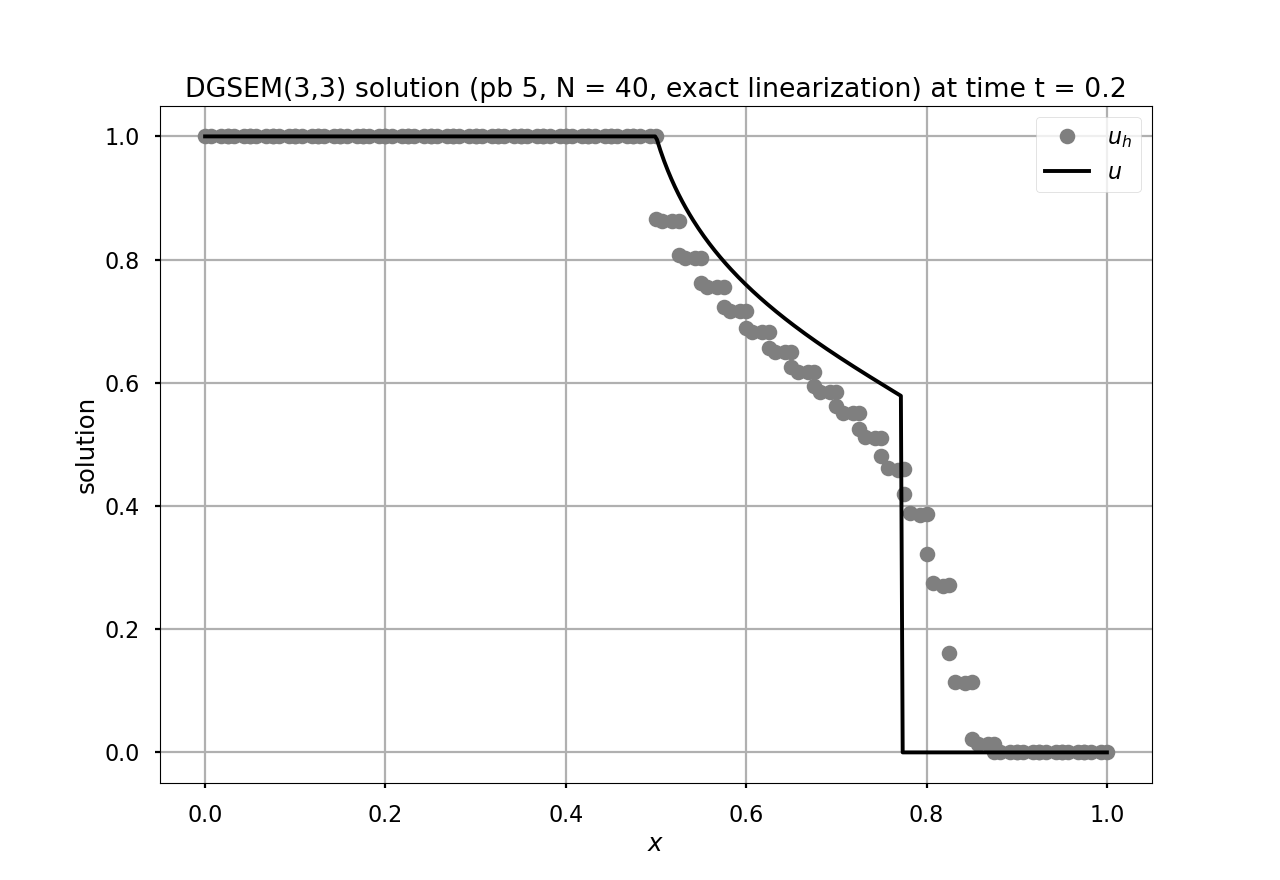}}%
\subfloat{\includegraphics[width=4.3cm]{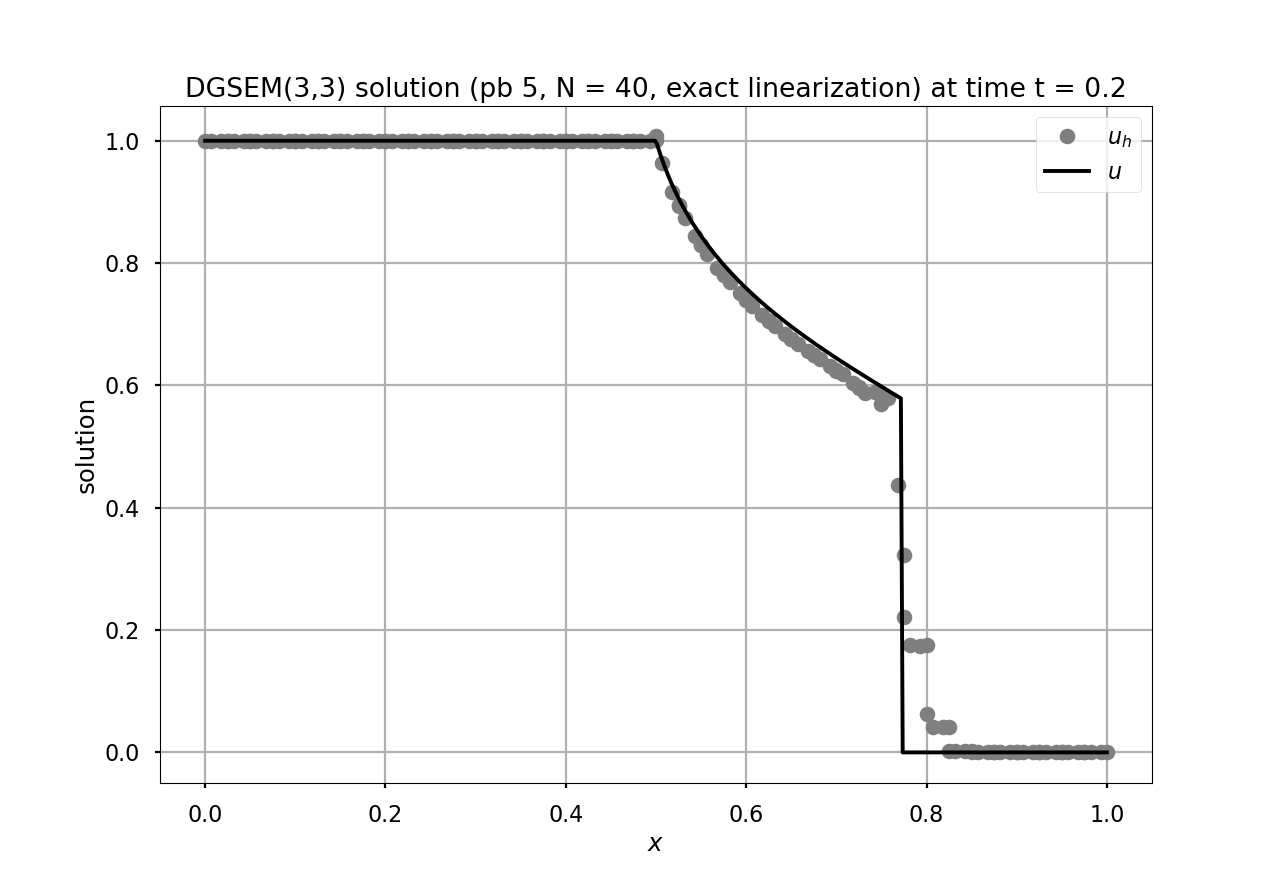}}\\ \vspace{-0.25cm}
\setcounter{subfigure}{0}
\subfloat[no graph viscosity]{\includegraphics[width=4.3cm]{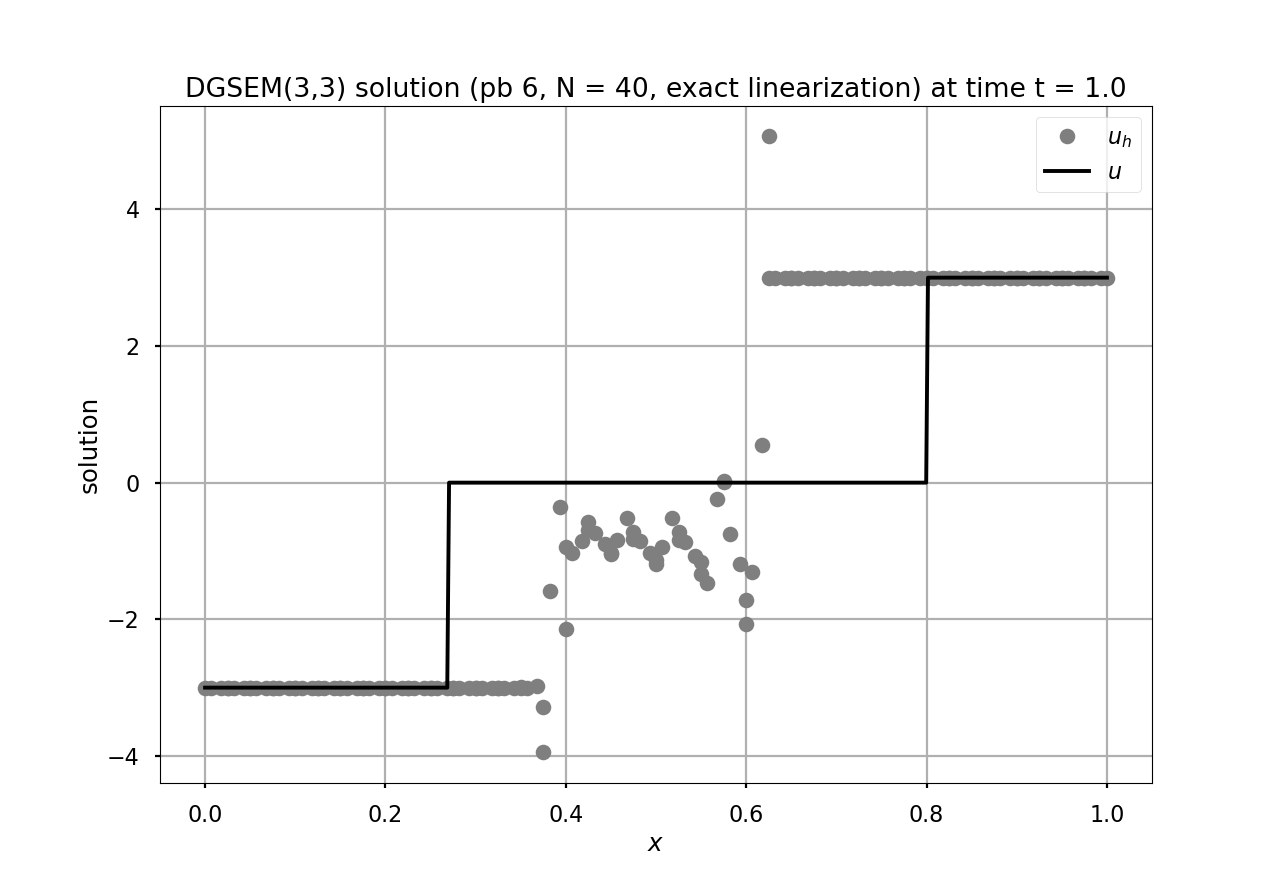}}
\subfloat[graph viscosity]{\includegraphics[width=4.3cm]{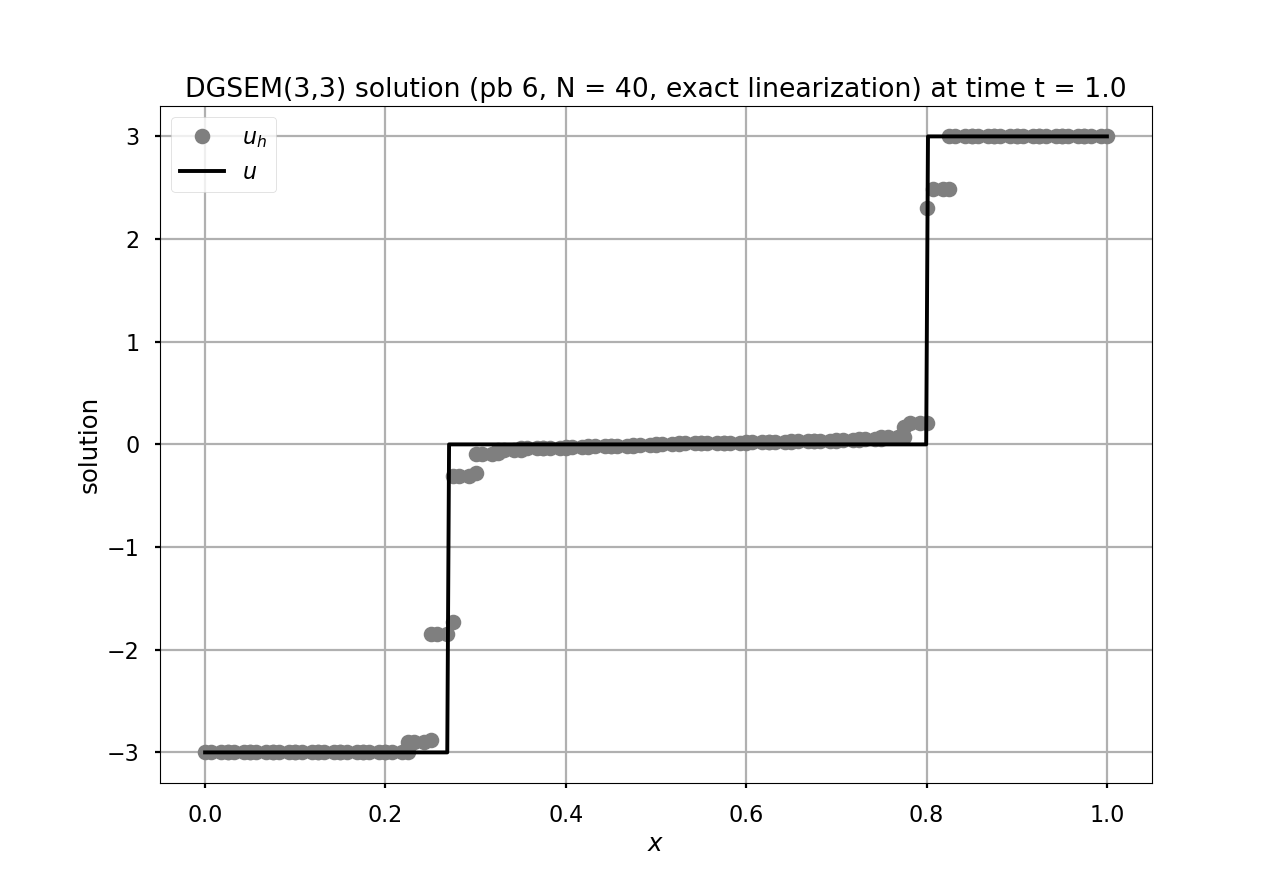}}
\subfloat[adapted graph viscosity]{\includegraphics[width=4.3cm]{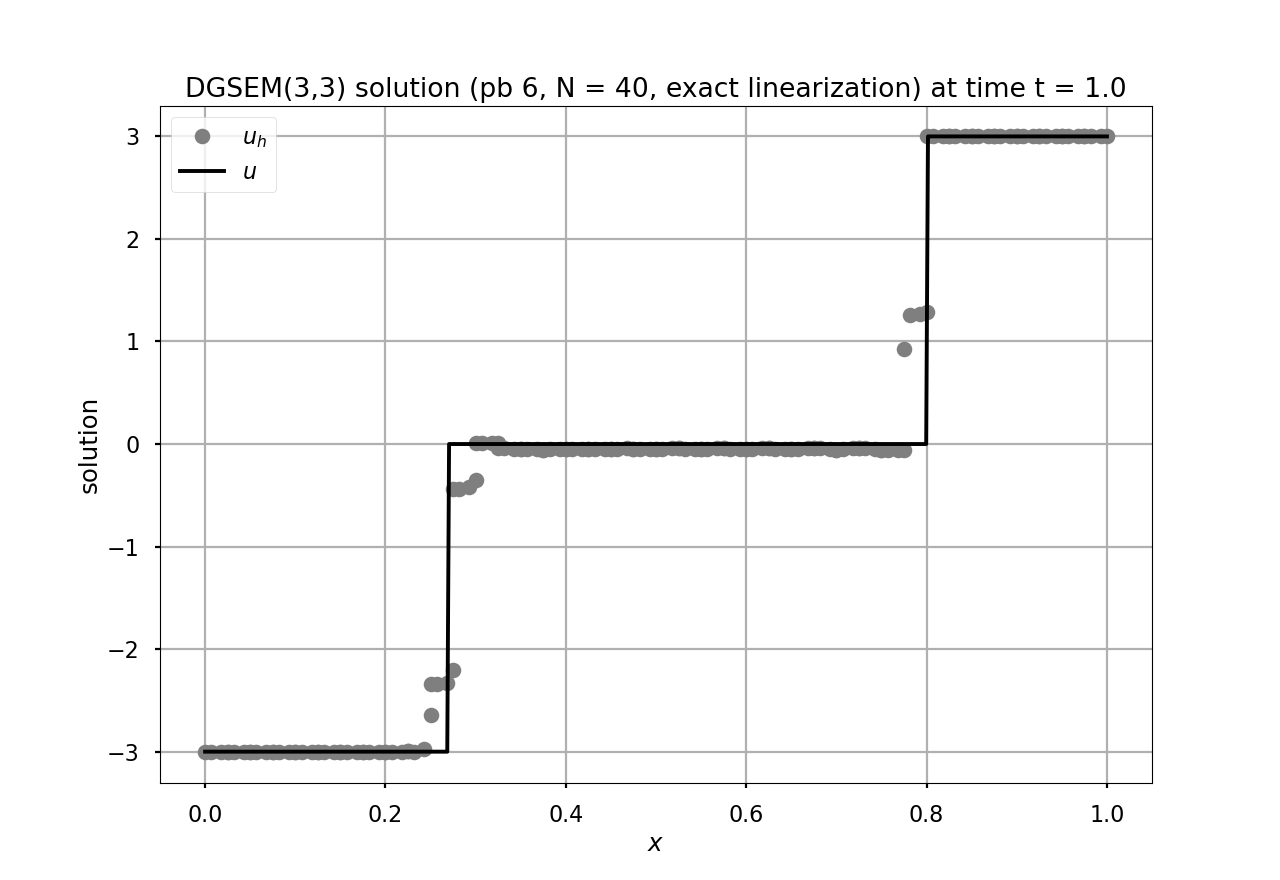}}
\caption{Buckley-Leverett equation: solutions to problems 4 (top) and 5 (bottom) defined in \cref{tab:def_pbs}. The $p+1=4$ DOFs per mesh element are displayed at the final time (bullets) and compared to the exact solution (lines).}
\label{fig:bukley_leverett}
\end{figure}

\section{Concluding remarks}\label{sec:conclusions}

We propose an analyze artificial viscosities in DGSEM schemes with implicit time stepping for the discretization of nonlinear scalar conservation laws in multiple space dimensions. We consider both a backward Euler time stepping for steady-state simulations, and a space-time DGSEM for time resolved simulations. The artificial dissipation is a first-order graph viscosity local to the space-time discretization elements. The schemes have no time step restriction, are MPP, satisfy a fully discrete inequality for every admissible convex entropy, and the associated discrete problems are well-posed. Numerical experiments in one space dimension are proposed to illustrate the properties of these schemes. Even being ES, the DGSEM without graph viscosity is not MPP and may fail to capture the entropy weak solution. For the sake of illustration, we also propose a local tuning of the artificial viscosity from using a troubled-cell indicator to reduce the dissipation, while keeping robustness and stability of the scheme. Future work will focus on the extension of this approach to systems of conservation laws. The use of the  present schemes with graph viscosity in the flux-corrected transport limiter to impose bounds on the high-order DGSEM solution is another possible direction of research.


%
%

%
%

%
%
\bibliographystyle{siamplain}
\bibliography{../../../BIBLIO_BIB/biblio_generale}

\end{document}